\documentclass[11pt]{article}

\usepackage{amsthm,amsmath,hyperref,txfonts}

\usepackage{color}
\usepackage{bbm}
\usepackage{graphicx,float}
\usepackage{authblk}
\usepackage[marginal]{footmisc}

\newcommand \nc{\newcommand}
\newtheorem{theorem}{Theorem}[section]
\newtheorem{lemma}[theorem]{Lemma}
\newtheorem{proposition}[theorem]{Proposition}
\newtheorem{corollary}[theorem]{Corollary}

\newtheorem{remark}[theorem]{Remark}

\nc{\ba}{\begin{array}}\nc{\ea}{\end{array}}
\nc{\be}{\begin{eqnarray}}\nc{\ee}{\end{eqnarray}}
\nc{\beq}{\begin{equation}}\nc{\eeq}{\end{equation}}
\nc{\bex}{\begin{eqnarray*}}\nc{\eex}{\end{eqnarray*}}
\nc{\btm}{\begin{theorem}} \nc{\etm}{\end{theorem}}
\nc{\blm}{\begin{lemma}} \nc{\elm}{\end{lemma}}

\newcommand{\sihao}{\fontsize{14pt}{\baselineskip}\selectfont}      

\nc{\R}{\mathbb{R}} \nc{\va}{\varepsilon} \nc{\ls}{\limits}

\def\pf{\noindent{\bf Proof.\quad}}

\allowdisplaybreaks

\begin{document}
\title{\sihao Optimal time-decay estimates for an Oldroyd-B model with zero viscosity}

\author[a]{Jinrui Huang}
\author[b]{Yinghui Wang}
\author[b]{Huanyao Wen\thanks{Corresponding author.\\ $\quad\quad\quad$ E-mail addresses: huangjinrui1@163.com (J. Huang), yhwangmath@163.com (Y. Wang), mahywen@scut.edu.cn (H. Wen), rzz@mail.ccnu.edu.cn (R. Zi).}}
\author[c]{Ruizhao Zi}
\affil[a]{School of Mathematics and Computational Science, Wuyi University, Jiangmen 529020, China. }
\affil[b]{School of Mathematics, South China University of Technology, Guangzhou 510641, China.}
\affil[c]{School of Mathematics and Statistics $\&$ Hubei Key Laboratory of Mathematical  Sciences,  Central China Normal University, Wuhan 430079, China.}

\date{}
\maketitle

\begin{abstract}
In this work, we consider the Cauchy problem for a diffusive Oldroyd-B model in three dimensions. Some optimal time-decay rates of the solutions are derived via analysis of upper and lower time-decay estimates provided that the initial data are small and that the absolute value of Fourier transform of the initial velocity is bounded below away from zero in a low-frequency region. It is worth noticing that the optimal rates are independent of the fluid viscosity or the diffusive coefficient, which is a different phenomenon from that for incompressible Navier-Stokes equations. 

\end{abstract}

{\bf\noindent keywords.} Oldroyd-B model, Cauchy problem, optimal time-decay estimates, zero viscosity.

{\noindent{\it\bf AMS Subject Classification (2010)}: 35Q35, 74H40, 76A10.}

\tableofcontents

\setcounter{section}{0} \setcounter{equation}{0}
\section{Introduction}

In this paper, we consider the Cauchy problem for a diffusive Oldroyd-B system describing the motion of viscoelastic fluids:
\begin{eqnarray} \label{system}
\begin{cases}
     \partial_tu+u\cdot\nabla u+\nabla p-\epsilon\Delta u=\kappa {\rm div}\tau,\\
     \partial_t\tau+u\cdot\nabla\tau-\mu\Delta\tau+\beta\tau=Q(\nabla u,\tau)+\alpha\mathbb{D}u,\\
     {\rm div}u=0,\\
     (u,\tau)(x,0)=(u_0,\tau_0),
\end{cases}
\end{eqnarray} on $\mathbb{R}^3\times (0,\infty)$. Here we denote by $u=(u_1,u_2,u_3)^\top:\mathbb{R}^3\rightarrow\mathbb{R}^3$ the velocity field of fluid, $\tau\in \mathbb{S}_3(\mathbb{R})$ is the purely elastic (the polymer) part of the stress tensor, $p\in\mathbb{R}$ is the pressure function of the fluid, $\mathbb{D}u=\frac12\left(\nabla u+\left(\nabla u\right)^\top\right)$ is the deformation tensor, and $$Q(\nabla u,\tau)=\Omega\tau-\tau\Omega+b(\mathbb{D}u \tau+\tau\mathbb{D}u)$$ admits the usual bilinear form  with the skew-symmetric part of velocity gradient $\Omega=\frac12\left(\nabla u-\left(\nabla u\right)^\top\right)$ and some $b\in [-1,1]$. The parameters $\epsilon$, $\kappa$, $\mu$, $\beta$ and $\alpha$ satisfy that $\epsilon, \mu\geq0, \kappa,\beta,\alpha>0$. 

Note that the diffusive term $\mu\Delta\tau$ corresponds to a spatial diffusion of the polymeric stresses where $\mu$ is usually called the center-of-mass diffusion coefficient. It is known that this coefficient is significantly smaller than other effects \cite{Bhave1}, and thus this term is usually ignored in the Oldroyd-B model (non-diffusive model). However, the center-of-mass diffusion effects can be physically justified 
{ 
 to model the shear and vorticity banding phenomena (\cite{Bhave2,Cates_2006,Dhont_2008,El-Kareh, Liu,Malek_etal_2018}). As showed by M\'{a}lek et al. \cite{Malek_etal_2018}, the stress diffusion term can be interpreted either as a consequence of a nonlocal energy storage mechanism or as a consequence of a nonlocal entropy production mechanism.} 
Thus the diffusive Oldroyd-B system has attracted much attention and been studied extensively, see \cite{Barrett-Boyava-M3-11, Barrett-Lu-Suli, Chupin, Constantin-Kliegl, Constantin-3, Elgindi-Liu, Elgindi-Rousset, La, Lu-Zhang, Wang-Wen}.

Our main goal here is to investigate the optimal time-decay estimates including the lower and upper estimates, and more importantly to study dependence of the optimal rates on the fluid viscosity $\epsilon$ and the center-of-mass diffusion coefficient $\mu$, which is divided into two cases, namely,
 \begin{itemize}
\item Case I: $\mu>0$ and $\epsilon\ge0$;

\item Case II: $\epsilon>0$ and $\mu\ge0$.
\end{itemize}
Without loss of generality, we assume that $\mu, \epsilon\le1$ throughout the rest of the paper. In what follows, we introduce some related works.

\subsection*{Non-diffusive Oldroyd-B model}
 The non-diffusive incompressible Oldroyd-B model (i.e., (\ref{system}) with $\mu=0$) was first proposed by Oldroyd in 1958 (\cite{Oldroyd}), which obeys an invariant constitutive law describing the general motion of an isotropic, incompressible visco-elastic liquid with significant non-Newtonion effects such as the Weissenberg climbing effect. An early existence result for
the model was obtained by Guillop${\rm \acute{e}}$ and Saut (\cite{Guillop-1}). They
obtained the global existence and uniqueness of strong solution with small initial data and a small coupling constant (like $\alpha$ in (\ref{system}) with $\mu=0$) in the Sobolev space $H^s(\Omega)$ where $\Omega$ is a bounded domain of $\mathbb{R}^3$. Fern\'andez-Cara, Guillen and Ortega (\cite{Fernandez-Guillpen-Ortega}) extended this results to the $L^p$-setting. Later on, the smallness restriction on the
coupling constant was removed by Molinet and Talhouk (\cite{Molinet-Talhouk}). In \cite{Hieber-Naito-Shibata}, Hieber, Naito and Shibata obtained a
global existence and uniqueness of the solution in exterior domains provided the initial data and the coupling constant are sufficiently
small. Fang, Hieber and Zi (\cite{Fang-Hieber-Zi}) extended the work \cite{Hieber-Naito-Shibata} to the case without any smallness assumption on the coupling constant. Note that the global existence theory for the non-diffusive model with arbitrarily large initial data is still open. However, for a special case that fluid flows do not stretch polymers but allow rotation, namely, $Q(\nabla u,\tau)=\Omega\tau-\tau\Omega$ in the non-diffusive model, Lions and Masmoudi (\cite{Lions-Masmoudi}) obtained the global existence of weak solution with arbitrarily large initial data. In this case, the model is usually called the corotational model which enjoys better {\it a priori} estimates due to the simplification of $Q(\nabla u,\tau)$.
{ 
For the upper-convected case (i.e.  $b = 1, ~Q(\nabla u,\tau) = \nabla u\tau + \tau (\nabla u^T)$), Hu and Leli\`{e}vre (\cite{Hu_Lelievre})
derived some new entropy estemates for the Oldroyd-B model and some related models. Concretely, they study the following nondimensional system, 
\begin{equation}\label{Oldroyd_B_nondim}
  \begin{cases}
    u_t + (u\cdot\nabla) u + \frac{1}{\mathbf{Re}} \nabla p - \frac{1-\omega}{\mathbf{Re}}\Delta u = \frac{1}{\mathbf{Re}}\mathrm{div} \tau,\\
    \tau_t + (u\cdot\nabla)\tau -\nabla u\tau -\tau (\nabla u)^{T} +\frac{1}{\mathbf{We}}\tau = \frac{\omega}{\mathbf{We}}(\nabla u + (\nabla u)^T),
  \end{cases}
\end{equation}
where $\omega\in (0,1)$ is a coupling constant,~$\mathbf{Re}$ is Reynold number and $\mathbf{We}$ is Weissenberg number.
Note that if we set $\epsilon = \frac{1-\omega}{\mathbf{Re}}, \kappa = \frac{1}{\mathbf{Re}}, \beta = \frac{1}{\mathbf{We}},\alpha = \frac{2\omega}{\mathbf{We}}$ and $\mu = 0,$ system \eqref{system} becomes \eqref{Oldroyd_B_nondim}.
By introducing the so-called conformation tensor, $A = \frac{We}{\omega}\tau + I$, the authors transform \eqref{Oldroyd_B_nondim} into the following system,
\begin{equation}\label{Oldroyd_B_nondim_A}
  \begin{cases}
    u_t + (u\cdot\nabla) u + \frac{1}{\mathbf{Re}} \nabla p - \frac{1-\omega}{\mathbf{Re}}\Delta u = \frac{\omega}{\mathbf{ReWe}}\mathrm{div}  A ,\\
    A_t + (u\cdot\nabla)A -\nabla uA -A (\nabla u)^{T}+\frac{1}{\mathbf{We}} (A-I)=0.
  \end{cases}
\end{equation}
We remark that \eqref{Oldroyd_B_nondim_A} can also be deduced from the dumbbell model
of polymeric fluids (refer to E, Li and Zhang (\cite{E_Li_Zhang_2004}) for instence).
Hu and Leli\`{e}vre (\cite{Hu_Lelievre}) obtained the following entropy estimate, for the Dirichlet problem in bounded domains,
\begin{equation}
  \begin{split}
  &\frac{\mathrm{d}}{\mathrm{d}t} F(u,A)+ (1-\omega)\int_D|\nabla u|^2 \mathrm{d}x+ \frac{\omega}{2 \mathbf{We}^2}\int_D \mathrm{tr}(A + A^{-1} -2 I)\mathrm{d}x = 0,
  \end{split}
\end{equation}
where 
\begin{equation*}F(u,A) = \frac{\mathbf{Re}}{2}\int_D |u|^2\mathrm{d}x + \frac{\omega}{2\mathbf{We}}\int_D (\mathrm{tr}A -\log(\mathrm{det}A) - 3)\mathrm{d}x
\end{equation*} is  the sum of the kinetic energy and the Helmholtz free energy (see equation $(3.13)$ in \cite{Hu_Lelievre}).
If one assume that the initial data $A_0$ is symmetric positive definite, it is easy to show that $A(t)$ is symmetric positive definite and $\mathrm{tr}A -\log(\mathrm{det}A) - 3\geq 0$. Using the fact that $\mathrm{tr}A -\log(\mathrm{det}A) - 3 = \mathrm{tr}(A -\log( A) - I) \leq \mathrm{tr}(A + A^{-1} -2 I)$ (see Lemma 1.1 and Lemma 2.1 in \cite{Boyaval_etal_2009} for detials), one can deduce the exponential time decay estimates for the free energy $F(u,A)$. Boyaval,  Leli\`evre and Mangoubi (\cite{Boyaval_etal_2009}) analyzed the  stability of some numerical schemes for system \eqref{Oldroyd_B_nondim_A}.
For the long-time behavior of the solution of system \eqref{system} with $\epsilon>0$ and $\mu = 0$ in $\mathbb{R}^3$, Hieber, Wen and Zi (\cite{Hieber-Wen-Zi}) obtained some upper time-decay estimates, and the decay rates are the same as those for the heat equation.
}

For the case without damping mechanism and the scaling invariant approach, see the works \cite{Zhu} and \cite{Chen, Fang-Zi, Fang-Zhang-Zi}, respectively. An energetic variational approach was first introduced by Lin, Liu, and Zhang (\cite{Lin-Liu-Zhang}) to understand the physical structure for the related systems (see for instance \cite{Cai, Hu-Lin, Hu-Wu, Lai, Lei1, Lei2, Lin, Lin-Zhang, Zhang-Fang} for more progress).

\subsection*{Diffusive Oldroyd-B model}
As mentioned in the beginning, the physical consideration for the diffusive Oldroyd-B model can be referred to \cite{Bhave1, Bhave2, El-Kareh, Liu, Malek_etal_2018}.
{ 
Inspired by the work of Rajagopal, Srinivasa (\cite{Rajagopal_2000}) and Ziegler, Wehrli (\cite{Ziegler_1987}), M\'{a}lek et al. (\cite{Malek_etal_2018}) derived variants of Oldroyd-B models with stress diffusion for viscoelastic rate-type fluids, using the specific Helmholtz free energy, the entropy production and the concept of evolving natural configuration, in both incompressible and compressible framework. We remark that the diffusive Oldroyd-B model can also be  derived  as a macroscopic closure of the Fokker–Planck–Navier–Stokes systems (refer to \cite{Barrett_Suli_2018} and \cite{La} for details).
}In the case $\epsilon>0$ and $\mu>0$, The existence of global-in-time weak solutions in two or three dimensions was obtained by Barrett, Boyaval, and S${\rm \ddot{u}}$li (\cite{Barrett-Boyava-M3-11, Barrett-Suli-M3-11}). For the regular solutions, Constantin and Kliegl (\cite{Constantin-Kliegl}) obtained the global existence and uniqueness of strong solutions to the two-dimensional model with varied polymer number density. La (\cite{La}) gave a rigorous derivation of the model in two dimensions as a macroscopic closure of a micro-macro model. 
{ 
By using a similiar free energy estimates in \cite{Hu_Lelievre}, Bathory, Bul\'{i}\v{c}ek and M\'{a}lek (\cite{Bathory_etal}) proved the existence of weak solutions for a generalized rate-type viscoelastic fluids with stress difusion in three-dimensional bounded domains.
Dostal\'{i}k, Pr\r{u}\v{s}a and Stein (\cite{Dostalik_etal_2020}) studied the large time behaviors for the related viscoelastic models with the free energy estimates in three-dimensional vessel.
}
For compressible case, please refer to \cite{Barrett-Lu-Suli, Lu-Zhang, Wang-Wen} for global existence of weak solutions, weak-strong uniqueness and conditional regularity, and optimal time-decay estimates, respectively.

In the case $\epsilon=0$ and $\mu>0$, i.e., the fluid viscosity vanishes, the global existence and uniqueness of regular solutions was proved by Elgindi and Rousset (\cite{Elgindi-Rousset}) in two dimensions provided that the initial data are arbitrary when $Q=Q(\nabla u,\tau)$ is ignored and that the initial data are small when $Q\not=0$. Later on, Elgindi and Liu (\cite{Elgindi-Liu}) extended the result in \cite{Elgindi-Rousset} to the three-dimensional case with small initial data in $H^3(\mathbb{R}^3)$-norm. Note that in this case the velocity of fluid obeys Euler system without damping whose global existence theory is unknown. \cite{Elgindi-Liu} and \cite{Elgindi-Rousset} demonstrate that the coupling effect of (\ref{system}) for $\epsilon=0$ and $\mu>0$ can make the solutions exist globally in time.

\bigskip

In summary, the previous works demonstrate what follows.
\begin{itemize}
\item For the non-diffusive Oldroyd-B model (i.e., $\mu=0$ and $\epsilon>0$), some global well-posedness results and upper time-decay rates have been obtained where the decay estimate is given by
    \be\label{HWZ}\|(\nabla^ku,\nabla^k\tau)(t)\|_{L^2(\mathbb{R}^3)}\le C(1+t)^{-\frac{3}{4}-\frac{k}{2}}\ee for $k=0,1,2$ and some known constant $C$ independent of $t$. However, the lower time-decay estimates are unknown.

\item For the diffusive Oldroyd-B model with zero viscosity (i.e., $\mu>0$ and $\epsilon=0$), there have been global well-posedness results and however time-decay estimates are unknown.
\end{itemize}


\subsection{Main results}
Denote by $\hat{f}(\xi)$ the Fourier transform of $f(x)$ throughout the paper. Now we are in a position to state our main results.

\begin{theorem}(Global existence)  \label{theorem} Assume that $(u_0,\tau_0)\in H^3(\mathbb{R}^3)$. For any given $\epsilon$ and $\mu$ satisfying the Case I or Case II, then there exists a sufficiently small constant $\varepsilon_0>0$  such that the Cauchy problem (\ref{system}) admits a unique global solution $(u^{\epsilon,\mu},\tau^{\epsilon,\mu})\in C([0,+\infty);H^3(\mathbb{R}^3))$ satisfying
\begin{equation}\label{1.2}
\begin{split}
     &\|u^{\epsilon,\mu}(t)\|_{H^3(\mathbb{R}^3)}^2+\int_0^t\left(\epsilon\|\nabla u^{\epsilon,\mu}(s)\|_{H^3(\mathbb{R}^3)}^2+\|\nabla u^{\epsilon,\mu}(s)\|_{H^2(\mathbb{R}^3)}^2\right)ds\leq C_1\varepsilon_0^2,
     \\&
     \|\tau^{\epsilon,\mu}(t)\|_{H^3(\mathbb{R}^3)}^2
     +\int_0^t\left(\mu\|\nabla\tau^{\epsilon,\mu}(s)\|_{H^3(\mathbb{R}^3)}^2+\|\tau^{\epsilon,\mu}(s)\|_{H^3(\mathbb{R}^3)}^2\right)ds\leq C_1\varepsilon_0^2
\end{split}
\end{equation} for $t\geq0$, provided that \begin{eqnarray*}
    \|u_0\|_{H^3(\mathbb{R}^3)}+\|\tau_0\|_{H^3(\mathbb{R}^3)}\leq \varepsilon_0,
\end{eqnarray*} where $\varepsilon_0$ is a constant depending on $\mu$ (or $\epsilon$) and some other known constants but independent of $\epsilon$ (or $\mu$) and $t$ for Case I (or for Case II), and the positive constant $C_1$ may depend on $\mu$ (or $\epsilon$) but independent of $\epsilon$ (or $\mu$) and $t$ for Case I (or for Case II).
\end{theorem}

\begin{theorem}(Optimal time-decay rates)\label{theorem_decay}
	For any given $\epsilon$ and $\mu$ satisfying the Case I or Case II, letting $(u^{\epsilon,\mu},\tau^{\epsilon,\mu})$ be the solution as in Theorem \ref{theorem}, then the following optimal time-decay estimates hold.

	\medskip

	\noindent {\em (i)} Assume that $ { (\hat{u}_0,\hat{\tau}_0)\in L^\infty(\mathbb{R}^3)}$. Then we have upper time-decay estimates of the solution as below:
	\begin{eqnarray}\label{opti1}
	\ \|\nabla^ku^{\epsilon,\mu}(t)\|_{L^2}\leq C_2(1+t)^{-\frac34-\frac{k}{2}},\ k=0,1,2,
	\end{eqnarray}
	and
	\begin{eqnarray}\label{opti2}
	\ \|\nabla^{k_1}\tau^{\epsilon,\mu}(t)\|_{L^2}\leq C_2(1+t)^{-\frac54-\frac{k_1}{2}},\ k_1=0,1
	\end{eqnarray}
	for any $t\geq0$ and generic positive constant $C_2$ which depends only on $\|(u_0,\tau_0)\|_{H^3}$, ${ \|(\hat{u}_0,\hat{\tau}_0)\|_{L^\infty_\xi}}$ and $C_1$.
	
\medskip

	\noindent {\em(ii)} Assume that $(\hat{u}_0,\hat{\tau}_0)\in L^\infty(\mathbb{R}^3) $ and in addition that $\inf_{0\leq |\xi|\leq R}|\hat{u}_0|\geq c_{0}>0$, for some $R>0$. Then there exist a positive time $t_0=t_0(\alpha,\kappa,\beta,\|(\hat{u}_0,\hat{\tau}_0)\|_{L^\infty_\xi})$ and a positive generic constant $c=c(\alpha,\kappa,\beta,c_0,C_2)$, such that
	\begin{eqnarray}\label{opti3}
	\|\nabla^ku^{\epsilon,\mu}(t)\|_{L^2}\geq c(1+t)^{-\frac34-\frac{k}{2}},\ k=0,1,2,
	\end{eqnarray}
	and
	\begin{eqnarray}\label{opti4}
	\|\nabla^{k_1}\tau^{\epsilon,\mu}(t)\|_{L^2}\geq c(1+t)^{-\frac54-\frac{k_1}{2}},\ k_1=0,1
	\end{eqnarray}
	for any $t\geq t_0$.

\end{theorem}

{ 
\begin{remark}
  Let $A = \tau + I$ be the conformation tensor, define the free energy entropic term by $ \int \mathrm{tr}(A - \log(A)- I)\mathrm{d}x$ (refer to \cite{Hu_Lelievre,Boyaval_etal_2009} for example).
  Note that for $\|\tau\|_{H^3}\ll 1,$
  \begin{equation*}
    \int \mathrm{tr}(A - \log(A) - I)\mathrm{d}x = \int \mathrm{tr}(\tau - \log (\tau+I))\mathrm{d}x =  \int \mathrm{tr}(\frac{\tau^2}{2} + O(\tau^3))\mathrm{d}x \simeq \|\tau\|_{L^2}^2,
  \end{equation*}
  (for the definition and properties of the matrix logarithm,  one can refer to Section 2.3 in \cite{Hall_2015}). Therefore, the entropic term $ \int \mathrm{tr}(A - \log(A)- I)\mathrm{d}x$ has the same decay estimate with $\|\tau\|_{L^2}^2.$ As a consequence, the free energy enjoys the same long time decay property with the kinetic energy.
\end{remark}
}
\begin{remark}
(\ref{opti1})-(\ref{opti4}) indicate that the time-decay rates are optimal and each of which is the same for both Case I and Case II, and that the decay speed of $\tau$ is faster than $u$ for the same order derivative. In particular, the optimal time-decay rates are independent of the fluid viscosity $\epsilon$ when $\mu>0$ is fixed, which is a different phenomenon from that for incompressible Navier-Stokes equations (cf. \cite{Schonbek1, Schonbek2} by Schonbek) and the incompressible Euler equations.
\end{remark}

\begin{remark} 
In Theorem \ref{theorem_decay}, we can not get optimal time-decay rates for the spatial derivative of the solution at the highest order in particular for Case I (zero viscosity), which is left open. However, for Case II the upper and lower time-decay rates of $\|\nabla^3u^{\epsilon,\mu}\|_{L^2}$ and $\|\nabla^2\tau^{\epsilon,\mu}\|_{L^2}$ can be improved to $-\frac94$ which are optimal (see Theorem \ref{theorem appendix}). This improves Hieber-Wen-Zi's result (see (\ref{HWZ})) in \cite{Hieber-Wen-Zi}.
\end{remark}
\begin{remark}\label{rem_u_0}
	As pointed out by  Schonbek, Schonbek and S\"uli (page 721 in \cite{Schonbek_Suli_1996}), there exists $u_0$ satisfying all the assumptions in Theorem \ref{theorem_decay}. Actually, for a given function $g\in C_0^\infty(\mathbb{R}^3)$ with $g(0) =2c_0\neq 0$. Define
	\begin{align}\label{def_u_0}
	\hat{u}_0(\xi)=(\hat{u}_{0,1},\hat{u}_{0,2},\hat{u}_{0,3})^\top:=\left(\frac{\xi_2}{\sqrt{\xi_1^2+\xi_2^2}}g(\xi),-\frac{\xi_1}{\sqrt{\xi_1^2+\xi_2^2}}g(\xi),0\right)^\top.
\end{align}
 Then,
	\begin{align*}
		\|u_0\|_{H^3}^2 =&\, \int_{\mathbb{R}^3}(1+|\xi|^2)^3|\hat{u}_0(\xi)|^2 d\xi = \int_{\mathbb{R}^3}(1+|\xi|^2)^3|g(\xi)|^2 d\xi < \infty, \\
		\widehat{\mathrm{div} u_0} =&\, i\xi_1\hat{u}_{0,1} +  i\xi_2\hat{u}_{0,2} + i\xi_3\hat{u}_{0,3} = i\frac{\xi_1\xi_2}{\sqrt{\xi_1^2+\xi_2^2}}g(\xi)-i\frac{\xi_1\xi_2}{\sqrt{\xi_1^2+\xi_2^2}}g(\xi) = 0,
     \end{align*}
	
	and
	$$\sup_{\xi\in\mathbb{R}^3}|\hat{u}_0| \leq \sup_{\xi\in\mathbb{R}^3}|g(\xi)|<\infty. $$
	Thus, $u_0\in H^3$ with $\sup_{\xi\in\mathbb{R}^3}|\hat{u}_0|<\infty$ and $\mathrm{div}u_0 = 0$.  Moreover, by the continuity of $g$, there exists a constant $R$, such that
	$$
		|g(\xi)| \geq c_0,~~\text{for all }~~0\leq |\xi|\leq R.
	$$
	Hence,
	$$
	\inf_{0\leq |\xi|\leq R}|\hat{u}_0|  = \inf_{0\leq |\xi|\leq R}\left(\frac{\xi_2^2}{\xi_1^2+\xi_2^2}|g(\xi)|^2 + \frac{\xi_1^2}{\xi_1^2+\xi_2^2}|g(\xi)|^2\right)^{\frac{1}{2}} = \inf_{0\leq |\xi|\leq R}|g(\xi)| \geq  c_0 >0.
	$$
	Therefore, $u_0$ defined by \eqref{def_u_0} satisfies all the assumptions  in Theorem \ref{theorem_decay}.
\end{remark}
\subsection{Ingredients of the proof}
The results in this paper consist of two parts: (1) The global well-posedness, and (2) the optimal decay rates of the solution.

For part (1), the main difficulty is to pursue the dissipation of the velocity $u$. In particular for Case I, due to the lack
of diffusion in $u$, the dissipation mechanism  hidden in the system must be shown in a proper way. In contrast with \cite{Elgindi-Liu,Elgindi-Rousset} where the authors detect a damping phenomenon on $\omega= {\rm curl}u$ and define a new quantity $\Gamma=\omega-\mathcal{R}\tau$ with $\mathcal{R}(\cdot)=-(-\Delta)^{-1}{\rm curl}({\rm div}(\cdot))$ to construct the energy estimates, our strategy is to first consider the auxiliary system of $(u, \sqrt{-\Delta}^{-1}\mathbb{P}\mathrm{div}\tau)$ (see \eqref{u-sigma} in Section 2), where $\mathbb{P}$ is the Larey projection operator. We are lucky to find that the auxiliary system \eqref{u-sigma} can handle Case I and Case II in a unified way, and the linear coupling term $\alpha\mathbb{D}u$ plays an import role in
the supplement of the dissipation of $u$. Furthermore, no additional conditions are required for any parameters. Thus we give a new proof for global existence in Case I compared with \cite{Elgindi-Liu}.

For part (2), to obtain the time decay rates, we study the Green matrix of system \eqref{system} in Fourier variables. One of the novel part of this paper is that we find that in both Case I and Case II, the decay rate of $\tau$ is faster than that of $u$ for each derivatives up to 1. This phenomenon has not been discovered in previous works \cite{Hieber-Wen-Zi,Wang-Wen}. Let us now explain this phenomenon more specifically in the linear level. As a matter of fact, according to Corollary \ref{coro}, the linear part of the solution $(u_L, \tau_L)$ can be represented as follows:
\begin{equation*}
\begin{split}
   \widehat{u_L}^j(\xi,t)
    =&\left(\mathcal{G}_3(\xi, t)-\epsilon|\xi|^2 \mathcal{G}_1(\xi, t)\right)\hat{u}^j_0(\xi)
    +i\kappa\mathcal{G}_1(\xi, t)\left(\delta_{j,k}-\frac{\xi_j\xi_k}{|\xi|^2}\right)\xi_l\hat{\tau}_0^{l,k}(\xi),
\end{split}
\end{equation*} and
\begin{equation*}
\begin{split}
   \widehat{\tau_L}^{j,k}(\xi,t)
   =&e^{-(\beta+\mu|\xi|^2)t} \hat{\tau}^{j,k}_0(\xi)+i\frac{\alpha}{2}\mathcal{G}_1(\xi, t)\left(\xi_k\hat{u}_0^j+\xi_j\hat{u}_0^k\right)
   \\&
   -\left(e^{-(\beta+\mu|\xi|^2)t}-\mathcal{G}_2(\xi, t)
   -\epsilon|\xi|^2 \mathcal{G}_1(\xi, t)\right)\frac{\xi_k\left(\delta_{j,p}-\frac{\xi_j\xi_p}{|\xi|^2}\right)+\xi_j\left(\delta_{k,p}-\frac{\xi_k\xi_p}{|\xi|^2}\right)}{|\xi|^2}\xi_l\hat{\tau}_0^{l,p}(\xi),
\end{split}
\end{equation*}
The definitions of $\mathcal{G}_1, \mathcal{G}_2$ and $\mathcal{G}_3$ can be found in \eqref{G}. A rough analysis on the eigenvalues $\lambda_\pm$ shows that $\mathcal{G}_1(\xi, t)$ and $\mathcal{G}_3(\xi, t)$ behave like the heat kernel in low frequencies (see Proposition \ref{theorem-Green}), thus
\begin{equation}\label{hat-u}
 |\widehat{u_L}(\xi,t)|\mathbbm{1}_{\{|\xi|\le R\}}
    \le Ce^{-\theta|\xi|^2t}\left(|\hat{u}_0(\xi)|+ |\hat{\tau}_0(\xi)|\right).
\end{equation}
As for $\widehat{\tau_L}$, a key observation is that, thanks to the damping $\beta\tau$ in $\eqref{system}_2$,
$\mathcal{G}_2(\xi, t)$ behaves better than the heat kernel in low frequencies (see\eqref{lemma_lower_eq2} for details). On this basis, one deduces that
\begin{equation}\label{hat-tau}
|\widehat{\tau}_L(\xi,t)|\mathbbm{1}_{\{|\xi|\le R\}}\le C\left(e^{-\beta t}+|\xi|e^{-\theta|\xi|^2t}\right)\left(|\hat{u}_0(\xi)|+ |\hat{\tau}_0(\xi)|\right).
\end{equation}
We can see from \eqref{hat-u} and \eqref{hat-tau} that the decay rates of $u_L$ and $\tau_L$ are determined by the factors $e^{-\theta|\xi|^2t}$ and $|\xi|e^{-\theta|\xi|^2t}$, respectively, and the extra factor $|\xi|$ in \eqref{hat-tau} provides time decay rate $(1+t)^{-\frac12}$. This explains why $\tau$ decays faster than $u$ of power $\frac12$.

On the other hand, the key observation  \eqref{lemma_lower_eq2} on $\mathcal{G}_2$ also plays an important role in the lower time-decay estimates (see Section 4.2). The main difficulty lies in that the nonlinear term $Q(\nabla u,\tau)$ does not decay fast enough to be absorbed by the linear term.  Similar difficulty appears in a previous result \cite{Hu-Wu_2013},  where the authors made use of the compatibility conditions to rewrite the main nonlinear terms into a gradient form, and thus gains one derivative. Although no compatibility conditions are available for our system \eqref{system},  the observation \eqref{lemma_lower_eq2} enable us to offset the insufficient time decay of the nonlinear nonlinear term $Q(\nabla u,\tau)$.

\bigskip

The rest of the paper are organized as follows. In Section \ref{sec2}, we give explicit expressions for the Fourier transform of Green matrix of the linearized system and an associated auxiliary system. Then  
the pointwise estimates on the low-frequency part of them are given, which are independent of the fluid viscosity and the center-of-mass diffusion coefficient. This will give some possibilities to achieve the desired results for the nonlinear system even if $\epsilon=0$ or $\mu=0$.
In Section \ref{sec3}, we present the global existence-uniqueness theorem for strong solutions with small initial data, and sketch the proof.
In Section \ref{sec4}, based on the energy estimates in Section \ref{sec3} and the pointwise estimates on the low-frequency part of the Green matrix in Section \ref{sec2}, and by virtue of the Duhamel's principle, we first get the time-decay estimates of the first and the second derivatives of the solutions in $L^2$ norm. Furthermore, we get $L^2$ decay estimates on the zero order and the third order derivatives of the solutions. Finally, we establish the lower bounds for the decay estimates by a delicate analysis on the eigenvalue.
In Appendix (Section \ref{sec5}), for Case II we give optimal time-decay estimates for $\|\nabla^3u^{\epsilon,\mu}\|_{L^2}$ and $\|\nabla^2\tau^{\epsilon,\mu}\|_{L^2}$ at the same rate. However, it seems difficult to get sharper decay rate for $\|\nabla^3\tau^{\epsilon,\mu}\|_{L^2}$, since the corresponding sharper decay for $\|\nabla^4u^{\epsilon,\mu}\|_{L^2}$ is usually needed in our approach (see (\ref{eq_tau})).

\section{Analysis of a linearized system}\label{sec2}\setcounter{equation}{0}

The main difficulty in exploring the optimal time-decay estimates of (\ref{system}) with $\epsilon$ or $\mu$ vanishing is the lack of dissipation. In order to supplement it, we introduce an auxiliary system. Namely, applying the Leray projection operator $\mathbb{P}$ and the operator $\mathbb{P}{\rm div}$ to $\eqref{system}_1$  and $\eqref{system}_2$, respectively, we obtain
\begin{eqnarray} \label{system-1}
\begin{cases}
     \partial_tu+\mathbb{P}\left(u\cdot\nabla u\right)-\epsilon\Delta u=\kappa \mathbb{P}{\rm div}\tau,\\
     \partial_t\mathbb{P}{\rm div}\tau+\mathbb{P}{\rm div}\left(u\cdot\nabla\tau\right)-\mu\Delta\mathbb{P}{\rm div}\tau+\beta\mathbb{P}{\rm div}\tau
     \\ \quad
     =\mathbb{P}{\rm div}Q(\nabla u,\tau)+\frac{\alpha}{2}\Delta u.
\end{cases}
\end{eqnarray}
Then, applying the operator $\Lambda^{-1}=(\sqrt{-\Delta})^{-1}$ to (\ref{system-1})$_2$ and denoting by
\begin{eqnarray*}
   \sigma\doteq \Lambda^{-1}\mathbb{P}{\rm div}\tau
\end{eqnarray*} with $\left(\hat{\sigma}\right)^j=i\left(\delta_{j,k}-\frac{\xi_j\xi_k}{|\xi|^2}\right)\frac{\xi_l}{|\xi|}\left(\hat{\tau}\right)^{l,k}$, one can rewrite (\ref{system-1}) as follows
\begin{eqnarray} \label{u-sigma}
\begin{cases}
     \partial_tu-\epsilon\Delta u-\kappa \Lambda\sigma=\mathcal{F}_1,\\
     \partial_t\sigma-\mu\Delta\sigma+\beta\sigma+\frac{\alpha}{2}\Lambda u=\mathcal{F}_2.
\end{cases}
\end{eqnarray} where the nonlinear terms are stated as below:
\begin{equation*}
\begin{split}
      &\mathcal{F}_1=-\mathbb{P}\left(u\cdot\nabla u\right),\
      \\ &
      \mathcal{F}_2=-\Lambda^{-1}\mathbb{P}{\rm div}\left(u\cdot\nabla\tau\right)+\Lambda^{-1}\mathbb{P}{\rm div}Q(\nabla u,\tau).
\end{split}
\end{equation*}
\subsection{Fourier transform of the Green matrix}
We consider the linearized system of (\ref{system}), i.e.,
\begin{eqnarray} \label{linear system}
\begin{cases}
     \partial_tu-\epsilon\Delta u=\kappa \mathbb{P}{\rm div}\tau,\\
     \partial_t\tau-\mu\Delta\tau+\beta\tau=\alpha\mathbb{D}u,\\
     (u,\tau)(x,0)=(u_0,\tau_0),
\end{cases}
\end{eqnarray}
and an auxiliary linear system which is the linearized equations of \eqref{u-sigma}:
\begin{eqnarray} \label{u-sigma_linear}
\begin{cases}
\partial_tu-\epsilon\Delta u-\kappa \Lambda\sigma=0,\\
\partial_t\sigma-\mu\Delta\sigma+\beta\sigma+\frac{\alpha}{2}\Lambda u=0,\\
(u,\sigma)(x,0)=(u_0,\sigma_0)(x),
\end{cases}
\end{eqnarray}where  $\sigma\doteq \Lambda^{-1}\mathbb{P}{\rm div}\tau$

Letting $\mathbb{G}_{u,\tau}$ and $\mathbb{G}_{u,\sigma}$ denote the Green matrixes of the systems (\ref{linear system}) and \eqref{u-sigma_linear}, respectively, then we have the following expressions of Fourier transform of them.
\begin{lemma}
Fourier transform of the solution to the auxiliary system (\ref{u-sigma_linear}) can be solved as follows:
\begin{eqnarray*}
    \left(\begin{matrix} \hat{u}\\ \hat{\sigma} \end{matrix}\right)=\hat{\mathbb{G}}_{u,\sigma}(\xi, t)\left(\begin{matrix} \hat{u}_0\\ \hat{\sigma}_0 \end{matrix}\right)
\end{eqnarray*}
with
\begin{eqnarray} \label{Green-matrix}
    \hat{\mathbb{G}}_{u,\sigma}(\xi, t)=\left(\begin{matrix} \big[\mathcal{G}_3(\xi,t)-\epsilon|\xi|^2\mathcal{G}_1(\xi,t)\big]\mathbb{I}_3 & \kappa|\xi|\mathcal{G}_1(\xi,t)\mathbb{I}_3 \\
    -\frac{\alpha}{2}|\xi|\mathcal{G}_1(\xi,t)\mathbb{I}_3 & \big[\mathcal{G}_2(t,\xi)+\epsilon|\xi|^2\mathcal{G}_1(\xi,t)\big]\mathbb{I}_3\end{matrix}\right),
\end{eqnarray}
and
\begin{eqnarray}\label{G}
   \mathcal{G}_1(\xi, t)=\frac{e^{\lambda_+t}-e^{\lambda_-t}}{\lambda_+-\lambda_-}, \ \mathcal{G}_2(\xi, t)=\frac{\lambda_+e^{\lambda_+t}-\lambda_-e^{\lambda_-t}}{\lambda_+-\lambda_-}, \ \mathcal{G}_3(\xi, t)=\frac{\lambda_+e^{\lambda_-t}-\lambda_-e^{\lambda_+t}}{\lambda_+-\lambda_-},
\end{eqnarray} where $\mathbb{I}_3$ is a $3\times3$ unit matrix.
\end{lemma}
\begin{proof}
Applying Fourier transform to the linearized system (\ref{u-sigma_linear}), we arrive at
\begin{eqnarray} \label{system-linear-Fourier}
\begin{cases}
     \partial_t\hat{u}^j-\kappa|\xi|\hat{\sigma}^j+\epsilon|\xi|^2\hat{u}^j=0,\\
     \partial_t\hat{\sigma}^j+\left(\mu|\xi|^2+\beta\right)\hat{\sigma}^j+\frac{\alpha}{2}|\xi| \hat{u}^j=0.
\end{cases}
\end{eqnarray} On one hand, it follows from the first equation of (\ref{system-linear-Fourier}) that
\begin{eqnarray} \label{sigmaj}
     \kappa|\xi|\hat{\sigma}^j=\partial_t\hat{u}^j+\epsilon|\xi|^2\hat{u}^j.
\end{eqnarray}
On the other hand, multiplying (\ref{system-linear-Fourier})$_2$ by $\kappa|\xi|$, one has
\begin{eqnarray} \label{sigmaj-2}
     \partial_t\left(\kappa|\xi|\hat{\sigma}^j\right)+\left(\mu|\xi|^2+\beta\right)\kappa|\xi|\hat{\sigma}^j+\frac{\alpha\kappa}{2}|\xi|^2 \hat{u}^j=0.
\end{eqnarray} Then, substituting (\ref{sigmaj}) into (\ref{sigmaj-2}), one obtains
\begin{eqnarray} \label{uj}
    \partial_{tt}\hat{u}^j+\left[\left(\mu+\epsilon\right)|\xi|^2+\beta\right]\partial_t\hat{u}^j+|\xi|^2\left[\epsilon\left(\mu|\xi|^2+\beta\right)+\frac{\alpha\kappa}{2}\right]\hat{u}^j=0,
\end{eqnarray}
which directly implies the corresponding characteristic equation:
\begin{eqnarray*}
    \lambda^2+\left[\left(\mu+\epsilon\right)|\xi|^2+\beta\right]\lambda+|\xi|^2\left[\epsilon\left(\mu|\xi|^2+\beta\right)+\frac{\alpha\kappa}{2}\right]=0,
\end{eqnarray*} with roots $\lambda_\pm$ satisfying
\begin{eqnarray} \label{eigen-relation}
\begin{cases}
     \lambda_{\pm}=\frac{-\left[\left(\mu+\epsilon\right)|\xi|^2+\beta\right]\pm\sqrt{\left[\left(\mu+\epsilon\right)|\xi|^2+\beta\right]^2-4|\xi|^2\left[\epsilon\left(\mu|\xi|^2+\beta\right)+\frac{\alpha\kappa}{2}\right]}}{2},\\
     \lambda_++\lambda_-=-\left[\left(\mu+\epsilon\right)|\xi|^2+\beta\right],\\
     \lambda_+\lambda_-=|\xi|^2\left[\epsilon\left(\mu|\xi|^2+\beta\right)+\frac{\alpha\kappa}{2}\right].
\end{cases}
\end{eqnarray}
Therefore, (\ref{uj}) can be solved in the following form.
\begin{eqnarray} \label{ode}
\begin{cases}
     {\hat{u}}^j(t)=c_1^je^{\lambda_+t}+c_2^je^{\lambda_-t},\\
     {\hat{u}}^j(0)={\hat{u}}^j_0,\\
     \partial_t{\hat{u}}^j(0)=\kappa|\xi|\hat{\sigma}^j_0-\epsilon|\xi|^2\hat{u}^j_0,
\end{cases}
\end{eqnarray} where
\begin{eqnarray*}
     c_1^j=\frac{\left(\lambda_-+\epsilon|\xi|^2\right)\hat{u}_0^j-\kappa|\xi|\hat{\sigma}_0^j}{\lambda_--\lambda_+},\
     c_2^j=\frac{-\left(\lambda_++\epsilon|\xi|^2\right)\hat{u}_0^j+\kappa|\xi|\hat{\sigma}_0^j}{\lambda_--\lambda_+}.
\end{eqnarray*} In conclusion, we obtain
\begin{eqnarray} \label{ode-s}
    \hat{u}^j(t)=\left(\frac{\lambda_+e^{\lambda_-t}-\lambda_-e^{\lambda_+t}}{\lambda_+-\lambda_-}-\epsilon|\xi|^2\frac{e^{\lambda_+t}-e^{\lambda_-t}}{\lambda_+-\lambda_-}\right)\hat{u}^j_0 +\kappa|\xi|\frac{e^{\lambda_+t}-e^{\lambda_-t}}{\lambda_+-\lambda_-}\hat{\sigma}_0^j.
\end{eqnarray} Similar calculations conducted on (\ref{system-linear-Fourier}) yield that
\begin{eqnarray} \label{ode-sigma}
    \hat{\sigma}^j(t)=-\frac{\alpha}{2}|\xi|\frac{e^{\lambda_+t}-e^{\lambda_-t}}{\lambda_+-\lambda_-}\hat{u}_0^j +\left[\frac{\lambda_+e^{\lambda_+t}-\lambda_-e^{\lambda_-t}}{\lambda_+-\lambda_-}+\epsilon|\xi|^2\frac{e^{\lambda_+t}-e^{\lambda_-t}}{\lambda_+-\lambda_-}\right]\hat{\sigma}_0^j.
\end{eqnarray}
The proof of the lemma is complete.
\end{proof}

From (\ref{ode-s}) and (\ref{ode-sigma}), we deduce the following corollary immediately.

\begin{corollary}\label{coro}
The explicit expression of $\hat{\mathbb{G}}_{u,\tau}(\xi,t)$ is determined by
\begin{equation*}
\begin{split}
   \hat{u}^j(\xi,t)
    =&\left(\frac{\lambda_+e^{\lambda_-t}-\lambda_-e^{\lambda_+t}}{\lambda_+-\lambda_-}-\epsilon|\xi|^2\frac{e^{\lambda_+t}-e^{\lambda_-t}}{\lambda_+-\lambda_-}\right)\hat{u}^j_0(\xi) \\&
    +i\kappa\frac{e^{\lambda_+t}-e^{\lambda_-t}}{\lambda_+-\lambda_-}\left(\delta_{j,k}-\frac{\xi_j\xi_k}{|\xi|^2}\right)\xi_l\hat{\tau}_0^{l,k}(\xi),
\end{split}
\end{equation*} and
\begin{equation*}
\begin{split}
   \hat{\tau}^{j,k}(\xi,t)
   =&e^{-(\beta+\mu|\xi|^2)t} \hat{\tau}^{j,k}_0(\xi)+i\frac{\alpha}{2}\frac{e^{\lambda_+t}-e^{\lambda_-t}}{\lambda_+-\lambda_-}\left(\xi_k\hat{u}_0^j+\xi_j\hat{u}_0^k\right)
   \\&
   -\left(e^{-(\beta+\mu|\xi|^2)t}-\frac{\lambda_+e^{\lambda_+t}-\lambda_-e^{\lambda_-t}}{\lambda_+-\lambda_-}
   -\epsilon|\xi|^2 \frac{e^{\lambda_+t}-e^{\lambda_-t}}{\lambda_+-\lambda_-}\right)
   \\&
   \times\frac{\xi_k\left(\delta_{j,p}-\frac{\xi_j\xi_p}{|\xi|^2}\right)+\xi_j\left(\delta_{k,p}-\frac{\xi_k\xi_p}{|\xi|^2}\right)}{|\xi|^2}\xi_l\hat{\tau}_0^{l,p}(\xi).
\end{split}
\end{equation*}
\end{corollary}
\begin{proof}
Conducting Fourier transform on both sides of (\ref{linear system})$_2$, we have
\begin{eqnarray*}
   \partial_t\hat{\tau}^{j,k}+\left(\beta+\mu|\xi|^2\right)\hat{\tau}^{j,k}=i\frac{\alpha}{2}\left(\xi_k\hat{u}^j+\xi_j\hat{u}^k\right),
\end{eqnarray*} which implies
\begin{equation} \label{tau}
\begin{split}
   \hat{\tau}^{j,k}=&e^{-(\beta+\mu|\xi|^2)t} \hat{\tau}^{j,k}_0+i\frac{\alpha}{2}e^{-(\beta+\mu|\xi|^2)t}\int_0^te^{(\beta+\mu|\xi|^2)s}\left(\xi_k\hat{u}^j+\xi_j\hat{u}^k\right)ds.
\end{split}
\end{equation}
Substituting (\ref{ode-s}) into the second term on the right-hand side of (\ref{tau}), and using (\ref{eigen-relation}), we have
\begin{equation}\label{2.22}
\begin{split}
   &\int_0^te^{(\beta+\mu|\xi|^2)s}\left(\xi_k\hat{u}^j+\xi_j\hat{u}^k\right)ds
   \\ =&
   \underbrace{\left(\xi_k\hat{u}_0^j+\xi_j\hat{u}_0^k\right)\int_0^te^{(\beta+\mu|\xi|^2)s}\left(\frac{\lambda_+e^{\lambda_-s}-\lambda_-e^{\lambda_+s}} {\lambda_+-\lambda_-}-\epsilon|\xi|^2\frac{e^{\lambda_+s}-e^{\lambda_-s}}{\lambda_+-\lambda_-}\right)ds}_{I_1},
   \\&
   +\underbrace{\kappa|\xi|\left(\xi_k{\hat{\sigma}}^j_0
   +\xi_j{\hat{\sigma}}^k_0\right)\int_0^te^{(\beta+\mu|\xi|^2)s}\frac{e^{\lambda_+s}-e^{\lambda_-s}}{\lambda_+-\lambda_-}
   ds}_{I_2},
\end{split}
\end{equation}
where
\begin{equation*}
\begin{split}
   I_1=&\left(\xi_k\hat{u}_0^j+\xi_j\hat{u}_0^k\right)\int_0^t\frac{\left(\lambda_++\epsilon|\xi|^2\right)e^{\lambda_-s}-\left(\lambda_-+\epsilon|\xi|^2\right) e^{\lambda_+s}}{\lambda_+-\lambda_-}e^{-\left(\lambda_++\lambda_-+\epsilon|\xi|^2\right)s}ds
   \\=&
   \left(\xi_k\hat{u}_0^j+\xi_j\hat{u}_0^k\right)\int_0^t\frac{\left(\lambda_++\epsilon|\xi|^2\right)e^{-\left(\lambda_++\epsilon|\xi|^2\right)s}-\left(\lambda_-+\epsilon|\xi|^2\right) e^{-\left(\lambda_-+\epsilon|\xi|^2\right)s}}{\lambda_+-\lambda_-}ds
   \\=&
   \left(\xi_k\hat{u}_0^j+\xi_j\hat{u}_0^k\right)\frac{\left[-e^{-\left(\lambda_++\epsilon|\xi|^2\right)s}+e^{-\left(\lambda_-+\epsilon|\xi|^2\right)s}\right]_0^t}{\lambda_+-\lambda_-}
   \\=&
   e^{-\epsilon|\xi|^2t}\frac{e^{-\lambda_-t}-e^{-\lambda_+t}}{\lambda_+-\lambda_-}\left(\xi_k\hat{u}_0^j+\xi_j\hat{u}_0^k\right),
\end{split}
\end{equation*} and thus
\begin{equation} \label{iI1}
\begin{split}
   i\frac{\alpha}{2}e^{-(\beta+\mu|\xi|^2)t}I_1&=i\frac{\alpha}{2}e^{-\left[\beta+(\mu+\epsilon)|\xi|^2\right]t}\frac{e^{-\lambda_-t}-e^{-\lambda_+t}}{\lambda_+-\lambda_-}\left(\xi_k\hat{u}_0^j+\xi_j\hat{u}_0^k\right)
   \\&
   =i\frac{\alpha}{2}e^{(\lambda_++\lambda_-)t}\frac{e^{-\lambda_-t}-e^{-\lambda_+t}}{\lambda_+-\lambda_-}\left(\xi_k\hat{u}_0^j+\xi_j\hat{u}_0^k\right)
   \\&=
   i\frac{\alpha}{2}\frac{e^{\lambda_+t}-e^{\lambda_-t}}{\lambda_+-\lambda_-}\left(\xi_k\hat{u}_0^j+\xi_j\hat{u}_0^k\right).
\end{split}
\end{equation}
The further calculations about the term $I_2$ are based on a key observation that
\begin{equation*}
\begin{split}
   \left(\lambda_++\epsilon|\xi|^2\right)\left(\lambda_-+\epsilon|\xi|^2\right)=\frac{\alpha\kappa}{2}|\xi|^2.
\end{split}
\end{equation*}
Indeed,
\begin{equation*}
\begin{split}
   I_2
   =&\kappa|\xi|\left(\xi_k{\hat{\sigma}}^j_0
   +\xi_j{\hat{\sigma}}^k_0\right)\int_0^te^{-(\lambda_++\lambda_-+\epsilon|\xi|^2)s}\frac{e^{\lambda_+s}-e^{\lambda_-s}}{\lambda_+-\lambda_-}
   ds
   \\=&
   \frac{2}{\alpha}\frac{\xi_k{\hat{\sigma}}^j_0
   +\xi_j{\hat{\sigma}}^k_0}{|\xi|}\left(\lambda_++\epsilon|\xi|^2\right)\left(\lambda_-+\epsilon|\xi|^2\right) \int_0^t\frac{e^{-\left(\lambda_-+\epsilon|\xi|^2\right)s}-e^{-\left(\lambda_++\epsilon|\xi|^2\right)s}}{\lambda_+-\lambda_-}
   ds
   \\=&
   \frac{2}{\alpha}\frac{\xi_k{\hat{\sigma}}^j_0
   +\xi_j{\hat{\sigma}}^k_0}{|\xi|}\frac{\left(\lambda_++\epsilon|\xi|^2\right)\left(1-e^{-\left(\lambda_-+\epsilon|\xi|^2\right)t}\right) -\left(\lambda_-+\epsilon|\xi|^2\right)\left(1-e^{-\left(\lambda_++\epsilon|\xi|^2\right)t}\right)}{\lambda_+-\lambda_-}
   \\=&
   \frac{2}{\alpha}\frac{\xi_k{\hat{\sigma}}^j_0
   +\xi_j{\hat{\sigma}}^k_0}{|\xi|}\frac{\lambda_+-\lambda_-+\left(\lambda_-+\epsilon|\xi|^2\right)e^{-\left(\lambda_++\epsilon|\xi|^2\right)t} -\left(\lambda_++\epsilon|\xi|^2\right)e^{-\left(\lambda_-+\epsilon|\xi|^2\right)t}}{\lambda_+-\lambda_-}
   \\=&
   \frac{2}{\alpha}\frac{\xi_k{\hat{\sigma}}^j_0
   +\xi_j{\hat{\sigma}}^k_0}{|\xi|}\left(1+e^{-\epsilon|\xi|^2t}\frac{\lambda_-e^{-\lambda_+t}-\lambda_+e^{-\lambda_-t}}{\lambda_+-\lambda_-}+\epsilon|\xi|^2 e^{-\epsilon|\xi|^2t}\frac{e^{-\lambda_+t}-e^{-\lambda_-t}}{\lambda_+-\lambda_-}\right).
\end{split}
\end{equation*}
Accordingly,
\begin{equation} \label{iI2}
\begin{split}
   i\frac{\alpha}{2}e^{-(\beta+\mu|\xi|^2)t}I_2
   =&i\frac{\xi_k{\hat{\sigma}}^j_0
   +\xi_j{\hat{\sigma}}^k_0}{|\xi|}\left(e^{-(\beta+\mu|\xi|^2)t}+e^{-\left[\beta+(\mu+\epsilon)|\xi|^2\right]t}\frac{\lambda_-e^{-\lambda_+t}-\lambda_+e^{-\lambda_-t}}{\lambda_+-\lambda_-}
   \right.
   \\ &\left.+\epsilon|\xi|^2 e^{-\left[\beta+(\mu+\epsilon)|\xi|^2\right]t}\frac{e^{-\lambda_+t}-e^{-\lambda_-t}}{\lambda_+-\lambda_-}\right)
   \\=&
   i\frac{\xi_k{\hat{\sigma}}^j_0
   +\xi_j{\hat{\sigma}}^k_0}{|\xi|}\left(e^{-(\beta+\mu|\xi|^2)t}+e^{\left(\lambda_++\lambda_-\right)t}\frac{\lambda_-e^{-\lambda_+t}-\lambda_+e^{-\lambda_-t}}{\lambda_+-\lambda_-}
   \right.
   \\ &\left.+\epsilon|\xi|^2 e^{\left(\lambda_++\lambda_-\right)t}\frac{e^{-\lambda_+t}-e^{-\lambda_-t}}{\lambda_+-\lambda_-}\right)
   \\=&
   i\frac{\xi_k{\hat{\sigma}}^j_0
   +\xi_j{\hat{\sigma}}^k_0}{|\xi|}\left(e^{-(\beta+\mu|\xi|^2)t}+\frac{\lambda_-e^{\lambda_-t}-\lambda_+e^{\lambda_+t}}{\lambda_+-\lambda_-}
   +\epsilon|\xi|^2 \frac{e^{\lambda_-t}-e^{\lambda_+t}}{\lambda_+-\lambda_-}\right).
\end{split}
\end{equation}
Collecting (\ref{ode-s}), (\ref{tau}), (\ref{2.22}), (\ref{iI1}) and (\ref{iI2}), and noticing the fact that
\begin{eqnarray*}
     \hat{\sigma}_0^j=i\left(\delta_{j,p}-\frac{\xi_j\xi_p}{|\xi|^2}\right)\frac{\xi_l}{|\xi|}\hat{\tau}_0^{l,p},
\end{eqnarray*} we finish the proof the corollary.
\end{proof}

\subsection{The low-frequency part}
\begin{proposition} \label{theorem-Green}  There exist positive constants $R=R(\alpha,\kappa,\beta)$, $\theta=\theta(\alpha,\kappa,\beta)$ and $K=K(\alpha,\kappa,\beta)$ such that
\begin{eqnarray}\label{pw}
     \left|\mathcal{G}_1(\xi, t)\right|,\left|\mathcal{G}_2(\xi, t)\right|,\left|\mathcal{G}_3(\xi,t)\right|,\left|\hat{\mathbb{G}}_{u,\sigma}(\xi,t)\right|,\left|\hat{\mathbb{G}}_{u,\tau}(\xi,t)\right|\leq Ke^{-\theta|\xi|^2t}
\end{eqnarray} hold for any $|\xi|\leq R$ and $t>0$.
\end{proposition}
\begin{proof}
 Denote
\begin{eqnarray*}
\mathfrak{D}(|\xi|)=\left[\left(\mu+\epsilon\right)|\xi|^2+\beta\right]^2-4|\xi|^2\left[\epsilon\left(\mu|\xi|^2+\beta\right)+\frac{\alpha\kappa}{2}\right].
\end{eqnarray*}
By a simple calculation, we obtain
\begin{equation}\label{Delta}
 	\begin{cases}
	 	\mathfrak{D}(|\xi|) \leq [(\mu+\epsilon)|\xi|^2 + \beta]^2\leq (2R^2+\beta)^2,\\[2mm]
	 	\mathfrak{D}(|\xi|)\geq\beta^2-4|\xi|^2(1+\beta+\frac{\alpha\kappa}{2})\geq \frac{\beta^2}{2}
 	\end{cases}
\end{equation}for $|\xi| \leq \displaystyle R = \min\Big\{1,\frac{\beta}{2\sqrt{2+2\beta+\kappa\alpha}}\Big\},$ where we use the assumption $\epsilon,\mu\le1$.
Next we rewrite $\lambda_+$ as
\begin{eqnarray}\label{expression-lamda+}
    \lambda_+=\frac{-2|\xi|^2\left[\epsilon\left(\mu|\xi|^2+\beta\right)+\frac{\alpha\kappa}{2}\right]}{(\mu+\epsilon)|\xi|^2+\beta +\sqrt{\mathfrak{D}(|\xi|)}}.
\end{eqnarray}
Since (\ref{Delta})$_1$ yields
\begin{equation*}
\begin{split}
    (\mu+\epsilon)|\xi|^2+\beta +\sqrt{\mathfrak{D}(|\xi|)}
    \leq 2R^2+\beta + \sqrt{(2R^2+\beta)^2}\leq 4R^2+2\beta,
\end{split}
\end{equation*}for $|\xi|\leq R$, then we have
\begin{equation} \label{lamda+upper}
\begin{split}
    \lambda_+\leq -\frac{2\left[\epsilon\left(\mu|\xi|^2+\beta\right)+\frac{\alpha\kappa}{2}\right]}{4R^2+2\beta}|\xi|^2\leq -\frac{\alpha\kappa}{4R^2+2\beta}|\xi|^2.
\end{split}
\end{equation}
Denote
$$\theta=\frac{\alpha\kappa}{2R_1^2+2\beta}.$$
Then by virtue of (\ref{lamda+upper}), we obtain the upper bound for $\lambda_+$ and $\lambda_-$, i.e.,
\begin{equation} \label{heat-decay}
\begin{split}
   \lambda_-\leq\lambda_+\leq -\theta|\xi|^2,\ {\rm and}\ \left|e^{\lambda_\pm t}\right|\leq e^{-\theta|\xi|^2t}
\end{split}
\end{equation} for all $|\xi|\leq R$.

By virtue of (\ref{Delta})$_2$, we have
\begin{equation*}
\begin{split}
   |\lambda_+-\lambda_-|=&\sqrt{\mathfrak{D}(|\xi|)}
   \geq \frac{\sqrt{2}}{2}\beta,
\end{split}
\end{equation*} for $|\xi|\leq R$, which, together with (\ref{heat-decay}), implies that
\begin{equation} \label{G1}
\begin{split}
   |\mathcal{G}_1(\xi,t)|\leq Ce^{-\theta|\xi|^2t}, \ {\rm for}\ {\rm all}\ |\xi|\leq R.
\end{split}
\end{equation}

Now we are in a position to estimate $\mathcal{G}_i(\xi,t)$ for $i=2,3$. We observe that
\begin{equation*}
\begin{split}
   \mathcal{G}_2(\xi,t)=\lambda_+\mathcal{G}_1(\xi,t)+e^{\lambda_-t}, \quad \mathcal{G}_3(\xi,t)=-\lambda_+\mathcal{G}_1(\xi,t)+e^{\lambda_+t}.
\end{split}
\end{equation*} and that
\begin{equation*}
\begin{split}
   |\lambda_+|\leq \frac{2|\xi|^2\left[\epsilon\left(\mu|\xi|^2+\beta\right)+\frac{\alpha\kappa}{2}\right]}{\beta}\leq \frac{2R^2\left(R^2+\beta+\frac{\alpha\kappa}{2}\right)}{\beta}.
\end{split}
\end{equation*} In addition, in view of (\ref{heat-decay}) and (\ref{G1}), we immediately obtain the estimates of $\mathcal{G}_i(\xi,t)$ for $i=2,3$. Then the upper bound of $|\hat{\mathbb{G}}_{u,\sigma}(\xi,t)|$ and $|\hat{\mathbb{G}}_{u,\tau}(\xi,t)|$ follows. The proof of Proposition \ref{theorem-Green} is complete.
\end{proof}

\section{Global existence and uniqueness}\label{sec3}
\setcounter{equation}{0}
The global existence and uniqueness of strong solutions for either the case ($\epsilon>0,\mu=0$) or the case ($\mu>0, \epsilon=0$) has been achieved in \cite{Hieber-Wen-Zi} and \cite{Elgindi-Liu}, respectively. In fact, it is not difficult to obtain the global existence and uniqueness of strong solution to the system \eqref{system} for the last case ($\epsilon>0,\mu>0$). For the sake of completeness and that some {\it a priori} estimates will be used next sections, we sketch the proof for all the cases (i.e., $\epsilon>0,\mu\ge0$ and $\mu>0, \epsilon\ge0$). To begin with, we have the following local existence and uniqueness result.
\begin{proposition}[Local existence]\label{localprop}
	 Suppose that $(u_0,\tau_0)\in H^3(\mathbb{R}^3)$, then there exists a positive constant $T_0 = T_0(\kappa,\alpha,\beta,\mu) $  for Case I ($\mu>0, \epsilon\geq 0$), or $T_0 = T_0(\kappa,\alpha,\beta,\epsilon)$ for Case II ($\epsilon>0,\mu\geq 0$), such that, the system \eqref{system} exists a unique local solution $(u^{\epsilon,\mu},\tau^{\epsilon,\mu})\in C([0,T_0];H^3(\mathbb{R}^3))$. 
\end{proposition}
\begin{proof}
		The proof can be done by using the standard contracting map theorem. Please refer for instance to \cite{Kawashima1983}.
\end{proof}

\bigskip

Next, we give some {\it a priori} estimates of
system \eqref{system}. The solutions usually should depend on $\epsilon$ and $\mu$. For brevity, we omit the superscripts throughout the rest of the paper.

Assume for the moment that
\begin{equation}\label{3.1}
\sup\limits_{0\leq t\leq T}
 \|(u,\tau)\|_{H^3}(t)\leq \delta,
\end{equation}
for some $T\in(0,T^*)$ where $T^*$ is the maximal time of
existence of the solutions as in Proposition \ref{localprop}, and the small constant $\delta\ge4\varepsilon_0>0$ is determined by (\ref{delta-epsilon}). The rest of this section tends to prove that
\begin{equation}\label{goal}
\sup\limits_{0\leq t\leq T}
 \|(u,\tau)\|_{H^3}(t)\leq \frac12\delta.
\end{equation}
Then, the global existence for system \eqref{system} can be deduced by using (\ref{goal}) and a standard continuity argument. With the regularity, the solution is unique.

\bigskip

Throughout the rest of the paper, we use the notation $``\cdot\lesssim \cdot"$ to represent that $``\cdot \le C
\cdot"$ for some generic known constants $C>0$ such as the constants yielded from the Sobolev's inequality and the Young's inequality. The notation $<\cdot,\cdot>$ represents the inner product in $L^2$ space.

\subsection{{\it A priori} estimate on the basic energy}
Firstly, applying $\nabla^k$ ($k=0,1,2,3$) to (\ref{system}), and then multiplying (\ref{system})$_1$ by $\alpha\nabla^k u$ and (\ref{system})$_2$ by $\kappa\nabla^k \tau$, we have from integration by parts that
\begin{equation} \label{energy}
\begin{split}
     &\frac12\frac{d}{dt}\left(\alpha\|u\|_{H^3}^2+\kappa\|\tau\|_{H^3}^2\right)+\alpha\epsilon\|\nabla u\|_{H^3}^2+\kappa\mu\|\nabla\tau\|_{H^3}^2+\kappa\beta\|\tau\|_{H^3}^2
     \\ =&
     -\alpha\sum_{k=0}^3<\nabla^k(u\cdot\nabla u),\nabla^k u>-\kappa\sum_{k=0}^3<\nabla^k(u\cdot\nabla \tau),\nabla^k\tau>
     \\ &
     +\kappa\sum_{k=0}^3<\nabla^kQ(\nabla u,\tau),\nabla^k\tau>=\sum_{p=1}^3\mathcal{N}_p,
\end{split}
\end{equation} where we have used the following cancellation structure due to the symmetry of $\tau$ and integration by parts,
\begin{equation*}
\begin{split}
    \alpha\kappa<\nabla^k\mathbb{D}u,\nabla^k\tau>=\alpha\kappa<\nabla^ku_{i,j},\nabla^k\tau^{ij}>=-\alpha\kappa<\nabla^ku,\nabla^k{\rm div}\tau>.
\end{split}
\end{equation*}

Now we are in a position to estimate $\mathcal{N}_k$ term by term. By using the H\"older inequality, estimates on the commutator and the Sobolev imbedding inequality that $\|f\|_{L^\infty}\lesssim \|\nabla f\|_{H^1}$, one has
\begin{equation*}
\begin{split}
    \mathcal{N}_1&=-\alpha\sum_{k=0}^3\left(<\nabla^k(u\cdot\nabla u),\nabla^k u>-<(u\cdot\nabla)\nabla^k u,\nabla^k u>\right)
    \\ & \lesssim
    \alpha\sum_{k=0}^3\left\|\nabla^k(u\cdot\nabla u)-(u\cdot\nabla)\nabla^k u\right\|_{L^2}\|u\|_{H^3}
    \lesssim
    \alpha\|\nabla u\|_{L^\infty}\|\nabla u\|_{H^2}\|u\|_{H^3}
    \\ & \lesssim
    \alpha\delta\|\nabla u\|_{H^2}^2.
\end{split}
\end{equation*}
Similarly,
\begin{equation*}
\begin{split}
    \mathcal{N}_2&=-\kappa\sum_{k=0}^3\left(<\nabla^k(u\cdot\nabla \tau),\nabla^k \tau>-<(u\cdot\nabla)\nabla^k \tau,\nabla^k \tau>\right)
    \\ & \lesssim
    \kappa\sum_{k=0}^3\left\|\nabla^k(u\cdot\nabla \tau)-(u\cdot\nabla)\nabla^k \tau\right\|_{L^2}\|\tau\|_{H^3}
    \\ & \lesssim
    \kappa\left(\|\nabla u\|_{L^\infty}\|\nabla \tau\|_{H^2}+\|\nabla\tau\|_{L^\infty}\|\nabla u\|_{H^2}\right)\|\tau\|_{H^3}
    \\ & \lesssim
    \kappa\delta\left(\|\nabla u\|_{H^2}^2+\|\nabla \tau\|_{H^2}^2\right).
\end{split}
\end{equation*} We remark that the estimates on $\mathcal{N}_1$ and $\mathcal{N}_2$ imply that we need some supplement of dissipation independent of $\epsilon$ or $\mu$.

The discussion on $\mathcal{N}_3$ will be more complicated, which is divided into two cases.

Case I (given $\mu>0$): Note that
\begin{equation*}
    \mathcal{N}_3=\kappa<Q(\nabla u,\tau),\tau>-\kappa\sum\limits_{k=1}^3<\nabla^{k-1}Q(\nabla u,\tau),\nabla^{k-1}\Delta\tau>.
\end{equation*} Then one has
\begin{equation*}
\begin{split}
    \mathcal{N}_3 &\lesssim \kappa\left(\|\nabla u\|_{L^\infty}\|\tau\|_{L^2}^2+\|Q(\nabla u,\tau)\|_{H^2}\|\nabla \tau\|_{H^3}\right)
    \\ &\lesssim
    \kappa\delta\|\tau\|_{L^2}^2+\left(\|\nabla u\|_{L^\infty}\|\tau\|_{H^2}+\|\tau\|_{L^\infty}\|\nabla u\|_{H^2}\right)\|\nabla \tau\|_{H^3}
    \\ & \lesssim
    \kappa\delta\|\tau\|_{L^2}^2+\kappa\delta\|\nabla u\|_{H^2}\|\nabla \tau\|_{H^3}
    \\ & \leq
    C\kappa\delta\|\tau\|_{L^2}^2+\frac14\kappa\mu\|\nabla \tau\|_{H^3}^2+C\mu^{-1}\delta^2\|\nabla u\|_{H^2}^2.
\end{split}
\end{equation*}

Case II (given $\epsilon>0$): By using the H\"older inequality and estimates on the commutator, one has
\begin{equation*}
\begin{split}
    \mathcal{N}_3&\lesssim \kappa\left(\|\nabla u\|_{L^\infty}\|\tau\|_{H^3}+\|\tau\|_{L^\infty}\|\nabla u\|_{H^3}\right)\|\tau\|_{H^3}
    \\ & \lesssim
    \kappa\|\tau\|_{H^3}^2\|\nabla u\|_{H^3}
    \lesssim
    \kappa\delta\|\tau\|_{H^3}\|\nabla u\|_{H^3}
    \leq
    \frac14\alpha\epsilon\|\nabla u\|_{H^3}^2+C\alpha^{-1}\epsilon^{-1}\kappa^2\delta^2\|\tau\|_{H^3}^2.
\end{split}
\end{equation*}

In conclusion, we have
\begin{equation} \label{basic-energy}
\begin{split}
     &\frac12\frac{d}{dt}\left(\alpha\|u\|_{H^3}^2+\kappa\|\tau\|_{H^3}^2\right)+\alpha\epsilon\|\nabla u\|_{H^3}^2+\kappa\mu\|\nabla\tau\|_{H^3}^2+\kappa\beta\|\tau\|_{H^3}^2
     \\ \leq &
     \begin{cases}
    C\delta(\alpha+\kappa+\mu^{-1}\delta)\|\nabla (u,\tau)\|_{H^2}^2+\frac14\kappa\mu\|\nabla \tau\|_{H^3}^2+C\kappa\delta\|\tau\|_{L^2}^2, \ {\rm for}\ {\rm given}\ \mu>0;\\
    C\delta(\alpha+\kappa)\|\nabla (u,\tau)\|_{H^2}^2
    +\frac14\alpha\epsilon\|\nabla u\|_{H^3}^2+C\alpha^{-1}\epsilon^{-1}\kappa^2\delta\|\tau\|_{H^3}^2, \ {\rm for}\ {\rm given}\ \epsilon>0.
    \end{cases}
\end{split}
\end{equation}

\subsection{Supplement of dissipation}
Applying the operator $\Lambda^k$ $(k=1,2,3)$ to the (\ref{u-sigma})$_1$, and multiplying the resulting equation by $\Lambda^{k-1}\sigma$, meanwhile, applying the operator $\Lambda^{k-1}$ to the (\ref{u-sigma})$_2$ and multiplying the resulting equation by $\Lambda^k u$, and then, summing up all the results, one has
\begin{equation} \label{supplement}
\begin{split}
    &\partial_t\sum_{k=1}^3<\Lambda^{k-1}\sigma,\Lambda^k u>+\kappa\sum_{k=1}^3\|\Lambda^k\sigma\|_{L^2}^2+\frac{\alpha}{2}\sum_{k=1}^3\|\Lambda^k u\|_{L^2}^2
    \\ =&
    \underbrace{-\sum_{k=1}^3<\Lambda^k\mathbb{P}\left(u\cdot\nabla u\right),\Lambda^{k-1}\sigma>}_{=\sum_{k=1}^3<\Lambda^{k-1}\mathbb{P}\left(u\cdot\nabla u\right),\Lambda^k\sigma>}+\underbrace{\epsilon\sum_{k=1}^3<\Delta\Lambda^ku,\Lambda^{k-1}\sigma>}_{=-\epsilon\sum_{k=1}^3<\Delta\Lambda^{k-1} u,\Lambda^k\sigma>}
    \\&
    -\sum_{k=1}^3<\Lambda^{k-2}\mathbb{P}{\rm div}\left(u\cdot\nabla\tau\right),\Lambda^k u>
    +\mu\sum_{k=1}^3<\Delta\Lambda^{k-1}\sigma,\Lambda^k u>
    \\&
    -\beta\sum_{k=1}^3<\Lambda^{k-1}\sigma,\Lambda^k u>
    +\sum_{k=1}^3<\Lambda^{k-2}\mathbb{P}{\rm div}Q(\nabla u,\tau),\Lambda^k u>
    \\ =&
    \sum_{i=1}^{6} J_i.
\end{split}
\end{equation}
By using the H\"older inequality and estimates on commutator, one has
\begin{equation*}
\begin{split}
    |J_1|+|J_3|+|J_6|
    \lesssim &\left(\|u\|_{L^\infty}\|\nabla u\|_{H^2}+\|\nabla u\|_{L^\infty}\|u\|_{H^2}\right)\|\nabla\sigma\|_{H^2}
    \\&
    +\left(\|u\|_{L^\infty}\|\nabla \tau\|_{H^2}+\|\nabla \tau\|_{L^\infty}\|u\|_{H^2}\right)\|\nabla u\|_{H^2}
    \\&
    +\left(\|\nabla u\|_{L^\infty}\|\tau\|_{H^2}+\|\tau\|_{L^\infty}\|\nabla u\|_{H^2}\right)\|\nabla u\|_{H^2}
    \\ \lesssim &
    \|u\|_{H^2}\|\nabla u\|_{H^2}\|\nabla \sigma\|_{H^2}+\|u\|_{H^2}\|\nabla\tau\|_{H^2}\|\nabla u\|_{H^2}
    \\ &
    +\|\nabla u\|_{H^2}\|\tau\|_{H^2}\|\nabla u\|_{H^2}
    \\ \lesssim &
    \delta\left(\|\nabla u\|_{H^2}^2+\|\nabla\tau\|_{H^2}^2\right),
\end{split}
\end{equation*} and
\begin{equation*}
\begin{split}
    &|J_2|+|J_4|+|J_5| \\
    \lesssim &\epsilon\|\nabla u\|_{H^3}\|\nabla\sigma\|_{H^2}+\mu\|\nabla\sigma\|_{H^3}\|\nabla u\|_{H^2}
    +\beta\|\sigma\|_{H^2}\|\nabla u\|_{H^2}
    \\ \leq &
    \epsilon_1\kappa\|\nabla\sigma\|_{H^2}^2+C\epsilon_1^{-1}\kappa^{-1}\epsilon\|\nabla u\|_{H^3}^2+\epsilon_1\alpha\|\nabla u\|_{H^2}^2+C\epsilon_1^{-1}\alpha^{-1}\mu\|\nabla\sigma\|_{H^3}^2
    \\ &
    +\epsilon_1\alpha\|\nabla u\|_{H^2}^2+C\epsilon_1^{-1}\alpha^{-1}\beta^2\|\sigma\|_{H^2}^2.
\end{split}
\end{equation*}
In conclusion, by choosing $\epsilon_1$ sufficiently small, one has
\begin{equation} \label{supplement-1}
\begin{split}
    &\partial_t\sum_{k=1}^3<\Lambda^{k-1}\sigma,\Lambda^k u>+\frac{\kappa}{2}\sum_{k=1}^3\|\Lambda^k\sigma\|_{L^2}^2+\frac{\alpha}{4}\sum_{k=1}^3\|\Lambda^k u\|_{L^2}^2
    \\ \lesssim &
    \kappa^{-1}\epsilon\|\nabla u\|_{H^3}^2+\alpha^{-1}\mu\|\nabla\tau\|_{H^3}^2
    +\alpha^{-1}\beta^2\|\tau\|_{H^2}^2
    +\delta\|\nabla(u,\tau)\|_{H^2}^2.
\end{split}
\end{equation}

\subsection{Total energy}
Combining the discussions in the previous two parts, one obtains the following total energy estimates from zero-order to third-order terms,
\begin{equation}\label{Total_energy}
\begin{split}
    & \frac{\rm d}{{\rm d}t}\mathcal{E}(t)+\alpha\epsilon\|\nabla u\|_{H^3}^2+\kappa\mu\|\nabla\tau\|_{H^3}^2+\kappa\beta\|\tau\|_{H^3}^2
    +\epsilon_2\sum_{i=1}^3\left(\frac{\kappa}{2}\|\Lambda^i\sigma\|_{L^2}^2+\frac{\alpha}{4}\|\Lambda^iu\|_{L^2}^2\right)
    \\ \leq &
    C\epsilon_2\left(\kappa^{-1}\epsilon\|\nabla u\|_{H^3}^2+\alpha^{-1}\mu\|\nabla\tau\|_{H^3}^2
    +\alpha^{-1}\beta^2\|\tau\|_{H^2}^2\right)
    \\+ &
    \begin{cases}
    C\delta(\alpha+\kappa+\mu^{-1}\delta+\epsilon_2)\|\nabla (u,\tau)\|_{H^2}^2+\frac14\kappa\mu\|\nabla \tau\|_{H^3}^2+C\kappa\delta\|\tau\|_{L^2}^2, \ {\rm for}\ {\rm given}\ \mu>0;\\
    C\delta(\alpha+\kappa+\epsilon_2)\|\nabla (u,\tau)\|_{H^2}^2+\frac14\alpha\epsilon\|\nabla u\|_{H^3}^2+C\alpha^{-1}\epsilon^{-1}\kappa^2\delta^2\|\tau\|_{H^3}^2, \ {\rm for}\ {\rm given}\ \epsilon>0.
    \end{cases}
\end{split}
\end{equation}
where
$$
\mathcal{E}(t)=\frac12\left(\alpha\|u\|_{H^3}^2+\kappa\|\tau\|_{H^3}^2\right)
  +\epsilon_2\sum_{k=1}^3<\Lambda^{k-1}\sigma,\Lambda^k u>= O(\tilde{\mathcal{E}}(t)),
$$ by using the Young's inequality for some small positive $\epsilon_2=\epsilon_2(\alpha,\kappa,\beta)$, and
$\tilde{\mathcal{E}}(t)=\|(u,\tau)\|_{H^3}^2$.

By further choosing suitable $\epsilon_2$ and $\delta$ which depend only on $\alpha$, $\kappa$, $\beta$ and $\epsilon$ for the case given $\epsilon>0$, and on $\alpha$, $\kappa$, $\beta$ and $\mu$ for the case given $\mu>0$, we have
\begin{equation*}
\begin{split}
    \frac{\rm d}{{\rm d}t}\mathcal{E}(t)\leq
    0.
\end{split}
\end{equation*} Then, by integrating the above inequality over $[0,t]$ and omitting some discussion on the generic constant for simplicity, one can get
\begin{equation*}
\begin{split}
    \left(\alpha\|u\|_{H^3}^2+\kappa\|\tau\|_{H^3}^2\right)(t)\leq 2\left(\alpha\|u_0\|_{H^3}^2+\kappa\|\tau_0\|_{H^3}^2\right)\leq 2(\alpha+\kappa)\varepsilon_0^2.
\end{split}
\end{equation*} Note also that
$$
\left(\alpha\|u\|_{H^3}^2+\kappa\|\tau\|_{H^3}^2\right)(t)\geq  \min\{\alpha,\kappa\}\|(u,\tau)\|_{H^3}(t).
$$
Therefore, choosing
\begin{equation}\label{delta-epsilon}
\frac{2(\alpha+\kappa)\varepsilon_0^2}{\min\{\alpha,\kappa\}}<\frac{\delta}{2}.
\end{equation}
The proof of (\ref{goal}) is complete.\qed

With (\ref{goal}) and (\ref{Total_energy}), it is easy to obtain (\ref{1.2}).

\section{Decay estimates for the nonlinear system}\label{sec4}
\setcounter{equation}{0}
The aim of this section is to establish the upper and lower decay rates of the solution to the system \eqref{system}. For simplicity,  let us denote $U=(u,\tau)^\top$.
\subsection{Upper time-decay estimates}
The proof will be achieved  in the following several steps.\par
\noindent{\bf Step 1: First order decay}\par
We first give the decay estimates for the low-frequency part of $U$.
\begin{lemma} \label{nonlinear-estimates}
For the nonlinear terms of (\ref{system}), we directly have the following estimates.
\begin{align} \label{F1}
   \|\mathcal{M}_1\|_{L^1}&\lesssim \|u\|_{L^2}\|\nabla u\|_{L^2},
   \\ \label{F2}
   \|\mathcal{M}_2\|_{L^1}&\lesssim  \|u\|_{L^2}\|\nabla\tau\|_{L^2}
   +\|\tau\|_{L^2}\|\nabla u\|_{L^2},\\ \label{F21}
      \|\mathcal{M}_1\|_{L^2}&\lesssim \|\nabla u\|_{H^1}\|\nabla u\|_{L^2},
   \\ \label{F22}
   \|\mathcal{M}_2\|_{L^2}&\lesssim  \|\nabla u\|_{H^1}\|\nabla\tau\|_{L^2}
   +\|\nabla \tau\|_{H^1}\|\nabla u\|_{L^2},
\end{align}
where
\begin{equation*}
\begin{split}
&\mathcal{M}_1=-\mathbb{P}\left(u\cdot\nabla u\right),\\
&\mathcal{M}_2=-u\cdot\nabla \tau+Q(\nabla u,\tau).
\end{split}
\end{equation*}
Moreover, the following time-decay estimate for the low-frequency part of the solution to the nonlinear system (\ref{system}), i.e.,
\begin{equation} \label{4.3}
\begin{split}
\left(
\int_{|\xi|\leq R}
 |\xi|^{2k}|\hat{U}|^2
  {\rm d}\xi
\right)^\frac{1}{2}
 \lesssim&
   (1+t)^{-\frac32(\frac12-\frac1q)-\frac{k}{2}}
    \|\hat{U}_0\|_{L^q_\xi}\\
 & +
  \int_0^t
  (1+t-s)^{-\frac32(\frac12-\frac1q)-\frac{k}{2}}
  \left\|\left(\mathcal{M}_1,\mathcal{M}_2\right)^\top(s)\right\|_{L^p}
  {\rm d}s,
\end{split}
\end{equation} holds for all $t>0$, where $q\geq 2, \frac{1}{p} + \frac{1}{q} = 1$ and $k\in[0,3]$.
\end{lemma}

\begin{proof} We only prove (\ref{4.3}).
From the Duhamel's principle, we have
\begin{equation} \label{duhamel}
U(t)
 =
  \mathbb{G}_{u,\tau}(t)\ast U_0
  +
  \int_0^t
   \mathbb{G}_{u,\tau}(t-s)\ast \left(\mathcal{M}_1,\mathcal{M}_2\right)^\top(s)
   {\rm d}s.
\end{equation}
It follows from Proposition \ref{theorem-Green} and Hausdorff-Young inequality that
\begin{equation*}
   \begin{split}
   &\left(\int_{|\xi|\leq R}|\xi|^{2k}|\hat{\mathbb{G}}_{u,\tau}|^2|\hat{U}_0|^2{\rm d}\xi\right)^{\frac12}
     \\ \lesssim&
    \|\hat{U}_0\|_{L^q_\xi}\left(\int_{|\xi|\leq R}(|\xi|^{2k}e^{-\theta|\xi|^2t})^\frac{q}{q-2}{\rm d}\xi\right)^{\frac{q-2}{2q}}
    \\ \lesssim& \|U_0\|_{L^p}\left(\int_{|\zeta|\leq R\sqrt{t}}(|\zeta|^{2k}t^{-k}e^{-\theta|\zeta|^2})^\frac{q}{q-2}t^{-\frac32}{\rm d}\zeta\right)^{\frac{q-2}{2q}}
    \\ \lesssim&
    \|U_0\|_{L^p}(1+t)^{-\frac32(\frac12-\frac1q)-\frac{k}{2}},
   \end{split}
\end{equation*}
where $q\geq 2, \frac{1}{p} + \frac{1}{q} = 1.$
Then, by similar calculations, we get
\begin{equation}\label{4.20}
\begin{split}
&\left(\int_{|\xi|\leq R}
 |\xi|^{2k}|\hat U(t)|^2
  {\rm d}\xi
\right)^\frac{1}{2}
\\  \lesssim&
  (1+t)^{-\frac32(\frac12-\frac1q)-\frac{k}{2}}
   \|U_0\|_{L^p}+
  \int_0^t
   (1+t-s)^{-\frac32(\frac12-\frac1q)-\frac{k}{2}}
    \left\|\left(\mathcal{M}_1,\mathcal{M}_2\right)^\top(s)\right\|_{L^p}
     {\rm d}s,
\end{split}
\end{equation}
where $q\geq 2, \frac{1}{p} + \frac{1}{q} = 1.$ The proof is complete.
\end{proof}

\bigskip

\begin{lemma}\label{lem-upper_bound1}
	Under the assumptions of Part (i) in Theorem \ref{theorem_decay}, it holds that
\begin{equation*}\sum\limits_{1\leq k\leq 3}\|\nabla^k(u,\tau)(t)\|_{L^2}\le C(1+t)^{-\frac{5}{4}}
  \end{equation*} for any $t\ge0$.
\end{lemma}
\begin{proof} It follows from a similar argument as (\ref{energy})  for the case $k=1,2,3$ and (\ref{supplement}) for the case $k=2,3$, one has,
\begin{equation*}
\begin{split}
    & \frac{\rm d}{{\rm d}t}\mathcal{H}(t)+\alpha\epsilon\|\nabla^2 u\|_{H^2}^2+\kappa\mu\|\nabla^2\tau\|_{H^2}^2+\kappa\beta\|\nabla\tau\|_{H^2}^2
    +\epsilon_3\sum_{i=2}^3\left(\frac{\kappa}{2}\|\Lambda^i\sigma\|_{L^2}^2+\frac{\alpha}{4}\|\Lambda^iu\|_{L^2}^2\right)
    \\ \leq &
    C\epsilon_3\left(\kappa^{-1}\epsilon\|\nabla^2 u\|_{H^2}^2+\alpha^{-1}\mu\|\nabla^2\tau\|_{H^2}^2
    +\alpha^{-1}\beta^2\|\nabla\tau\|_{H^2}^2\right)
    \\+ &
    \begin{cases}
    C\delta(\alpha+\kappa+\mu^{-1}\delta+\epsilon_3)\|\nabla (u,\tau)\|_{H^2}^2+\frac14\kappa\mu\|\nabla^2 \tau\|_{H^2}^2, \ {\rm given}\ \mu>0;\\
    C\delta(\alpha+\kappa+\epsilon_3)\|\nabla (u,\tau)\|_{H^2}^2+\frac14\alpha\epsilon\|\nabla^2 u\|_{H^2}^2+C\alpha^{-1}\epsilon^{-1}\kappa^2\delta^2\|\nabla\tau\|_{H^2}^2, \ {\rm given}\ \epsilon>0.
    \end{cases}
\end{split}
\end{equation*}
where
$$
\mathcal{H}(t)=\frac12\left(\alpha\|\nabla u\|_{H^2}^2+\kappa\|\nabla\tau\|_{H^2}^2\right)
  +
  \epsilon_3\sum_{k=2}^3<\Lambda^{k-1}\sigma,\Lambda^k u>= O(\tilde{\mathcal{H}}(t)),
$$ by using the Young's inequality for some small positive $\epsilon_3=\epsilon_3(\alpha,\kappa,\beta)$, and
$\tilde{\mathcal{H}}(t)=\|\nabla(u,\tau)\|_{H^2}^2$.

By further choosing suitable $\epsilon_3$ and $\delta$ which depend only on $\alpha$, $\kappa$, $\beta$ and $\epsilon$ for the case given $\epsilon>0$, and on $\alpha$, $\kappa$, $\beta$ and $\mu$ for the case given $\mu>0$, we have
\begin{equation}\label{E4.21}
\begin{split}
\frac{\rm d}{{\rm d}t} \mathcal{H}(t)+\frac12\left(\frac{\kappa\beta}{2}\|\nabla\tau\|_{H^2}^2+\frac{\epsilon_3\alpha}{4}\sum_{i=2}^3\|\Lambda^iu\|_{L^2}^2\right)
\lesssim
  \delta
   \|\nabla u\|_{L^2}^2.
\end{split}
\end{equation}
Using the Plancherel's theorem, and splitting the integral into two
parts, we have
\begin{equation*}
\|\Lambda^2u\|_{L^2}^2=\|\nabla^2u\|_{L^2}^2
 =
  \int_{|\xi|\leq R}
   |\xi|^4|\hat{u}|^2
    {\rm d}\xi
  +
  \int_{|\xi|\geq R}
   |\xi|^4|\hat{u}|^2
    {\rm d}\xi
 \geq
   C
   \int_{|\xi|\geq R}
    |\xi|^2|\hat{u}|^2
     {\rm d}\xi
\end{equation*}
for some given radius $R>0$.
Similarly, we have
\begin{equation*}
\begin{split}
\|\Lambda^3u\|_{L^2}^2
 \geq
  C\int_{|\xi|\geq R}
    |\xi|^2|\hat{u}|^2
     {\rm d}\xi.
\end{split}
\end{equation*}
Then, from \eqref{E4.21}, we have
\begin{equation}\label{4.32}
\begin{split}
&
\frac{\rm d}{{\rm d}t}
   \mathcal{H}(t)
  +\frac14\left(\frac{\kappa\beta}{2}\|\nabla\tau\|_{H^2}^2+\frac{\epsilon_3\alpha}{4}\sum_{i=2}^3\|\Lambda^iu\|_{L^2}^2\right)
  +\frac{\epsilon_3\alpha}{8}\int_{|\xi|\geq R}
    |\xi|^2|\hat{u}|^2
     {\rm d}\xi
  \\
\lesssim
 &
  \delta\left(
   \int_{|\xi|\leq R}
    |\xi|^2|\hat{u}|^2
     {\rm d}\xi
   +
   \int_{|\xi|\geq R}
    |\xi|^2|\hat{u}|^2
     {\rm d}\xi
 \right).
\end{split}
\end{equation}
Due to (\ref{4.32}) and the smallness of $\delta$, we obtain
\begin{equation*}
\begin{split}
&
\frac{\rm d}{{\rm d}t}
   \mathcal{H}(t)
  +\frac14\left(\frac{\kappa\beta}{2}\|\nabla\tau\|_{H^2}^2+\frac{\epsilon_3\alpha}{4}\sum_{i=2}^3\|\Lambda^iu\|_{L^2}^2\right)
\\&
  +\frac{\epsilon_3\alpha}{16}\int_{|\xi|\geq R}
    |\xi|^2|\hat{u}|^2
     {\rm d}\xi
\lesssim
  \delta
   \int_{|\xi|\leq R}
    |\xi|^2|\hat{u}|^2
     {\rm d}\xi.
\end{split}
\end{equation*} Adding $\frac{\epsilon_3\alpha}{16}\int_{|\xi|\leq R}|\xi|^2|\hat{u}|^2{\rm d}\xi$ on both sides of (\ref{4.32}), one has
\begin{equation}\label{4.34}
\begin{split}
\frac{\rm d}{{\rm d}t}
 \mathcal{H}(t)
  +D\tilde{\mathcal{H}}(t)
\lesssim
   \int_{|\xi|\leq R}
    |\xi|^2|\hat u(t)|^2
     {\rm d}\xi
\end{split}
\end{equation} for some positive constant $D=\min\left\{\frac{\kappa\beta}{4},\frac{\epsilon_3\alpha}{16}\right\}$.
Define
\begin{equation*}
\mathcal M(t)
 =
  \sup\limits_{0\leq \tau\leq t}
   (1+s)^{\frac52}
    \tilde{\mathcal{H}}(s).
\end{equation*}
Notice that $\mathcal M(t)$ is non-decreasing and
\begin{equation*}
\sqrt{\tilde{\mathcal{H}}(s)}
 \lesssim
  (1+s)^{-\frac{5}{4}}
   \sqrt{\mathcal M(t)},\ \ \ \ 0\leq s\leq t.
\end{equation*}
Then, we get from Lemma \ref{nonlinear-estimates} and the definition of $\tilde{\mathcal{H}}(t)$ that
\begin{equation}\label{4.35}
\begin{split}
&
\left(
\int_{|\xi|\leq R}
 |\xi|^2|\hat U(t)|^2
  {\rm d}\xi
\right)^\frac{1}{2}\\
 \lesssim
  &
  (1+t)^{-\frac{5}{4}}\|\hat{U}_0\|_{L^\infty_\xi}
   +
    \delta\int_0^t
   (1+t-\tau)^{-\frac{5}{4}}
    (1+\tau)^{-\frac{5}{4}}
     {\rm d}\tau
      \sqrt{\mathcal M(t)}\\
 \lesssim
  &
   (1+t)^{-\frac{5}{4}}
    \left(
     \|U_0\|_{L^1}
     +
     \delta\sqrt{\mathcal M(t)}
    \right).
\end{split}
\end{equation}
From (\ref{4.34}) and (\ref{4.35}), we have
\begin{equation}\label{4.36}
\begin{split}
\tilde{\mathcal{H}}(t)
 \lesssim
 &
  {\rm e}^{-D t}
   \tilde{\mathcal{H}}(0)
  +
  \int_0^t
   {\rm e}^{-D(t-\tau)}
    \left(
     \int_{|\xi|\leq R}
      |\xi|^2
       |\hat{U}(\tau)|^2
        {\rm d}\xi
    \right)
   {\rm d}\tau\\
 \lesssim
  &
   {\rm e}^{-Dt}
   \tilde{\mathcal{H}}(0)
   +
   \int_0^t
    {\rm e}^{-D(t-\tau)}
     (1+\tau)^{-\frac52}
      \left[
       \|\hat{U}_0\|_{L^\infty_\xi}^2
       +
        \delta^2\mathcal M(t)
      \right]
       {\rm d}\tau\\
 \lesssim
  &
   (1+t)^{-\frac52}
    \left[
     \tilde{\mathcal{H}}(0)
     +
     \|U_0\|_{L^1}^2
     +
       \delta^2\mathcal M(t)
    \right].
\end{split}
\end{equation}
Using the definition of $\mathcal{M}(t)$ and the smallness of $\delta$, we
have
\begin{equation*}
\mathcal M(t)
 \leq
  C\left(\tilde{\mathcal{H}}(0)+
  \|\hat{U}_0\|_{L^\infty_\xi}^2\right),
\end{equation*}
which implies that
\begin{equation}\label{decay1}
\tilde{\mathcal{H}}^{\frac12}(t)
 \lesssim
  (1+t)^{-\frac{5}{4}}
  \left(
  \tilde{\mathcal{H}}^{\frac12}(0)
  +
  \|\hat{U}_0\|_{L^\infty_\xi}
  \right)
   \leq
   C_2(1+t)^{-\frac{5}{4}}.
\end{equation}
The proof of Lemma \ref{lem-upper_bound1} is complete.
\end{proof}
\noindent{\bf Step 2: Zero order and the second order decay}
\begin{lemma}\label{lem-upper_bound2}
	Under the assumptions of Part (i) in Theorem \ref{theorem_decay}, it holds that
\begin{equation}
\|(u,\tau)(t)\|_{L^2}\le C(1+t)^{-\frac{3}{4}},
  \end{equation}
 and
 \begin{equation}
\sum_{2\le k\le3}\|\nabla^k(u,\tau)(t)\|_{L^2}\le C(1+t)^{-\frac{7}{4}},
  \end{equation}
   for any $t\ge0$.
\end{lemma}
\begin{proof}
By similar discussion as (\ref{4.34}), one obtains
\begin{equation}\label{zero-2}
\begin{split}
\frac{\rm d}{{\rm d}t}
 \mathcal{H}_0(t)
  +D_0\|U(t)\|_{H^{3}}^2
\lesssim
   \int_{|\xi|\leq R}
    |\hat u(t)|^2
     {\rm d}\xi,
\end{split}
\end{equation} for some positive constant $D_0=D_0(C_2)$ and
$$
\mathcal{H}_0(t)=\frac12\left(\alpha\|u\|_{H^3}^2+\kappa\|\tau\|_{H^3}^2\right)
  +
  \epsilon_4\sum_{k=1}^3<\Lambda^{k-1}\sigma,\Lambda^k u>
  = O(\|U(t)\|_{H^{3}}^2),
$$ by using the Young's inequality for some small positive $\epsilon_4=\epsilon_4(C_2)$. Then, it follows from Lemma \ref{nonlinear-estimates} and (\ref{decay1}) that
\begin{equation*}
\begin{split}
\|U(t)\|_{H^{3}}^2
 \lesssim
 &
  {\rm e}^{-D_0 t}
   \|U_0\|_{H^{3}}^2
  +
  \int_0^t
   {\rm e}^{-D_0(t-\tau)}
    \left(
     \int_{|\xi|\leq R}
      |\hat{U}(\tau)|^2
        {\rm d}\xi
    \right)
   {\rm d}\tau\\
 \lesssim
  &
   {\rm e}^{-D_0t}
   \|U_0\|_{H^{3}}^2
   +
   \int_0^t
    {\rm e}^{-D_0(t-\tau)}
     (1+\tau)^{-\frac32}
      \|\hat{U}_0\|_{L^\infty_\xi}^2
       {\rm d}\tau
       \\ &
       +\int_0^t {\rm e}^{-D_0(t-\tau)}\left(\int_0^\tau(1+\tau-s)^{-\frac34}(1+s)^{-\frac54}{\rm d}s\right)^2{\rm d}\tau
       \\
 \leq
  &
   C_2(1+t)^{-\frac32}.
\end{split}
\end{equation*}
In conclusion, we have
\begin{eqnarray} \label{U-L2}
    \|U(t)\|_{L^2}\leq C_2(1+t)^{-\frac34}.
\end{eqnarray}

For the second order decay, by a similar discussion as (\ref{4.34}), one obtains
\begin{equation}\label{2ed_decay_1}
\begin{split}
\frac{\rm d}{{\rm d}t}
\mathcal{H}_1(t)
+D_1\|\nabla^2(u,\tau)\|_{H^{1}}^2
\lesssim
\int_{|\xi|\leq R}
|\xi|^4|\hat u(t)|^2
{\rm d}\xi,
\end{split}
\end{equation} for some positive constant $D_1=D_1(C_2)$ and
$$
\mathcal{H}_1(t)=\alpha\|\nabla^2u\|_{H^1}^2+\kappa\|\nabla^2\tau\|_{H^1}^2
+
\epsilon_5<\Lambda^{2}\sigma,\Lambda^3 u>
= O(\|\nabla^2(u,\tau)(t)\|_{H^{1}}^2),
$$ where we have used the Young's inequality for some small positive $\epsilon_5=\epsilon_5(C_2)$. Then, it follows from Lemma \ref{nonlinear-estimates} and (\ref{2ed_decay_1}) that
\begin{equation}\label{4.3632}
\begin{split}
\|\nabla^2(u,\tau)(t)\|_{H^{1}}^2
\lesssim
&
{\rm e}^{-D_1 t}
\|U_0\|_{H^{3}}^2
+
\int_0^t
{\rm e}^{-D_1(t-s)}
\left(
\int_{|\xi|\leq R}
|\xi|^4|\hat{u}(s)|^2
{\rm d}\xi
\right)
{\rm d}s\\
\lesssim
&
{\rm e}^{-D_1t}
\|U_0\|_{H^{3}}^2
+
\int_0^t
{\rm e}^{-D_1(t-s)}
(1+s)^{-\frac72}
\|\hat{U}_0\|_{L^\infty_\xi}^2
{\rm d}s
\\ &
+\int_0^t {\rm e}^{-D_1(t-s)}\int_0^s(1+s-s^\prime)^{-\frac72}(1+s^\prime)^{-\frac82}{\rm d}s^\prime{\rm d}s
\\
\leq
&
C_2(1+t)^{-\frac72}.
\end{split}
\end{equation}
Combining  \eqref{U-L2} with \eqref{4.3632}, we complete the proof of Lemma \ref{lem-upper_bound2}.
\end{proof}
\noindent{\bf Step 3: Further decay estimates}\par
\begin{lemma}\label{lem-upper_bound3}
	Under the assumptions of Part (i) in Theorem \ref{theorem_decay}, it holds that
\begin{equation}\label{tau01}
		 \|\nabla^k \tau\|_{L^2}  \leq C (1+ t)^{-\frac{5}{4} - \frac{k}{2}},\quad k=0,1,
	 \end{equation}  for any $t\ge0$.
\end{lemma}
\begin{proof}
 Applying $\nabla^k$ ($k=0,1$) to (\ref{system})$_2$, multiplying the result by $\nabla^k \tau$, and integrating with respect to $x$, we have from Lemmas \ref{lem-upper_bound1} and \ref{lem-upper_bound2}  that
	 \begin{equation}\label{eq_tau}
			\begin{split}
				 &\frac{1}{2}\frac{\mathrm{d}}{\mathrm{d}t} \|\nabla^k \tau\|_{L^2}^2 + \mu \|\nabla^{k+1} \tau \|_{L^2}^2 + \frac{\beta}{2}\|\nabla^k \tau\|_{L^2}^2\\
				 \lesssim & \alpha \|\nabla^{k+1}u\|_{L^2}^2 + \|\nabla^k{Q}(\nabla u,\tau)\|_{L^2}^2 + \|\nabla^{k}(u\cdot \nabla\tau)\|_{L^2}^2\\
				 \lesssim & \alpha \|\nabla^{k+1}u\|_{L^2}^2 + \|\nabla^{k+1}u\|_{L^2}^2\|\tau\|_{L^\infty}^2 + \|\nabla u\|_{L^2}^2\|\nabla^{k}\tau\|_{L^\infty}^2\\&+ \|\nabla^{k+1}\tau\|_{L^2}^2\|u\|_{L^\infty}^2 + \|\nabla \tau\|_{L^2}^2\|\nabla^{k}u\|_{L^\infty}^2\\
				 \lesssim & (1+ t)^{-(\frac{5}{2} + k)},
			\end{split}
	 \end{equation}
	 which implies that
	 \begin{equation*}
		 \frac{\mathrm{d}}{\mathrm{d}t} \|\nabla^k \tau\|_{L^2}^2 +  \beta\|\nabla^k \tau\|_{L^2}^2 \lesssim (1+ t)^{-(\frac{5}{2} + k)},\text{ for } k=0,1.
	 \end{equation*}
	 Using Gronwall inequality, we obtain
	 \begin{equation}\label{eq_tau_2}
		 \|\nabla^k \tau\|_{L^2}^2 \leq e^{-\beta t}\|\nabla^k \tau_0\|_{L^2}^2 + \int_{0}^{t} e^{-\beta(t-s)}(1+s)^{-(\frac{5}{2} + k)}\mathrm{d} s \leq C (1+ t)^{-(\frac{5}{2} + k)},
	 \end{equation}
for $k=0,1. $	This completes the proof of Lemma \ref{lem-upper_bound3}.
	\end{proof}
Now the proof of Part (i) in Theorem \ref{theorem_decay} is accomplished by Lemmas \ref{lem-upper_bound1}-\ref{lem-upper_bound3}.

\subsection{Lower time-decay estimates}
To prove Part (ii) in Theorem \ref{theorem_decay}, we begin this subsection with giving some refined estimates for the Fourier transform of the Green function.

\begin{lemma}\label{lemma_lower}
Let  $R$ and $\theta$ be the constants chosen in Proposition \ref{theorem-Green}. There exist three positive constants $\eta=\eta(\alpha,\beta,\kappa)$,  $c_1=c_1(\alpha,\kappa,\beta)$ and $\tilde{c}_1=c_1(\alpha,\kappa,\beta)$, and a positive time $t_1=t_1(\alpha,\kappa,\beta)$, such that
	\begin{equation}\label{lemma_lower_eq1}
	|\mathcal{G}_1(\xi,t)| \geq c_1 e^{-\eta |\xi|^2 t},~~|\mathcal{G}_3(\xi,t)| \geq c_1 e^{-\eta |\xi|^2 t}, \ {\rm for}\ {\rm all}\ |\xi| \leq R \ {\rm and}\ t\geq t_1.
	\end{equation}
	Moreover, $\mathcal{G}_2$ admits the following refined estimate
	\begin{equation}\label{lemma_lower_eq2}
	|\mathcal{G}_2(\xi,t)|
	\leq \tilde{c}_1\left(|\xi|^2 e^{-\theta|\xi|^2t} + e^{-\frac{\beta t}{2}}\right), \ {\rm for}\ {\rm all}\ |\xi| \leq R \ {\rm and}\ t\geq t_1.
	\end{equation}
\end{lemma}
\begin{proof}	
Noting first that $\lambda_+ - \lambda_-=\sqrt{\mathfrak{D}(|\xi|)}$, based on the analysis in the proof of Proposition \ref{theorem-Green},  for all $|\xi|\leq R$, there holds
	\begin{equation} \label{lower-0}
	\frac{\sqrt{2}}{2}\beta\le\lambda_+ - \lambda_-
	\leq  2R^2 + \beta.
	\end{equation}
 Therefore, it follows from (\ref{expression-lamda+}) that
	\begin{equation*}
	\lambda_+\geq -\frac{2}{\beta}|\xi|^2\left[\epsilon\left(\mu|\xi|^2+\beta\right)+\frac{\alpha\kappa}{2}\right]\geq -\frac{2}{\beta}\left(R^2+\beta+\frac{\alpha\kappa}{2}\right)|\xi|^2 =: -\eta |\xi|^2,
	\end{equation*}
where $\eta=\frac{2}{\beta}\left(R^2+\beta+\frac{\alpha\kappa}{2}\right)$. This, together with (\ref{expression-lamda+}), implies that, for any $t\geq t_1\doteq\frac{\ln 2}{2R^2+\beta}$,
	\begin{equation} \label{lower-2}
	|e^{\lambda_+t}-e^{\lambda_-t}| = \left|e^{\lambda_+ t}\big(1 - e^ {-(\lambda_+-\lambda_-)t}\big)\right|   \geq  \frac{1}{2}e^{-\eta |\xi|^2 t},
	\end{equation}
	and
	\begin{equation} \label{lower-3}
	|\lambda_+e^{\lambda_-t}-\lambda_-e^{\lambda_+t}| = |e^{\lambda_+ t}\big(\lambda_+ e^ {-(\lambda_+-\lambda_-)t} - \lambda_-\big)| \geq |\lambda_+-\lambda_-|e^{-\eta |\xi|^2t}.
	\end{equation}
Accordingly, thanks to \eqref{G}, (\ref{lower-0}), (\ref{lower-2}) and (\ref{lower-3}), choosing
   \begin{equation} \label{lower3}
	{c}_1=\min\left\{\frac{1}{2(2R^2 + \beta)},1\right\},
  \end{equation}
we finish the proof of \eqref{lemma_lower_eq1}.

	Next, we infer from (\ref{eigen-relation})$_1$ and (\ref{expression-lamda+}) that, for $|\xi|\leq R$,
	\begin{equation} \label{lower-4}
	\begin{split}
			\lambda_- &\leq -\frac{\beta}{2},~~ |\lambda_-|\leq (\mu + \epsilon)|\xi|^2/2+\beta\leq R^2+\beta,\\
			\lambda_+ &\leq -\frac{\alpha\kappa}{2R^2+2\beta}|\xi|^2,~~ |\lambda_+|\leq \frac{2|\xi|^2(R^2+\beta+\alpha\kappa)}{\beta}.
	\end{split}
	\end{equation}
	Then it follows from (\ref{lower-0}) and (\ref{lower-4}) that
	\begin{equation*}
    \begin{split}
	|\mathcal{G}_2|	
    \leq
    \left|\frac{\lambda_+}{\lambda_+-\lambda_-}\right|e^{\lambda_+t}+\left|\frac{\lambda_-}{\lambda_+-\lambda_-}\right|e^{\lambda_-t}
    \leq
    \tilde{c}_1\left(|\xi|^2 e^{-\theta|\xi|^2t}+e^{-\frac{\beta t}{2}}\right),
	\end{split}
    \end{equation*} where
    \begin{equation*}
       \tilde{c}_1=\max\left\{\frac{2\sqrt{2}(2R^2+\beta+\alpha\kappa)}{\beta^2},\frac{\sqrt{2}(R^2+\beta)}{\beta}\right\}.
    \end{equation*}
	The proof of (\ref{lemma_lower_eq2}) is complete.
\end{proof}
Then, we have the lower decay bounds for the linear system. In the following, denote by $$\mathbb{G}_{u,\sigma}(t) :=(\mathbb{G}_{u,\sigma}^1(t),\mathbb{G}_{u,\sigma}^2(t))^\top.$$

\begin{lemma}\label{linear_lower_bound}
	Under the assumptions of Part (ii) in Theorem \ref{theorem_decay}, there exist a positive time $t_2$ and a positive generic constant $c_2$, such that
	\begin{equation}\label{linear_lower_bound_1}
	\|\nabla^k\left(\mathbb{G}_{u,\sigma}^1(t)\ast U_0^\sigma\right)\|_{L^2} \geq c_2(1 + t)^{-\frac{3}{4}-\frac{k}{2}},
	\end{equation}
	\begin{equation}\label{linear_lower_bound_2}
	\|\nabla^k\left(\mathbb{G}_{u,\sigma}^2(t)\ast U_0^\sigma\right)\|_{L^2} \geq c_2(1 + t)^{-\frac{5}{4}-\frac{k}{2}},
	\end{equation}
for all $t\geq t_2$, and $k=0, 1, 2, \cdot\cdot\cdot$, where $t_2=t_2(\alpha,\kappa,\beta,\|(\hat{u}_0,\hat{\tau}_0)\|_{L^\infty_\xi})$, $c_2=c_2(\alpha,\kappa,\beta,c_0,\|(\hat{u}_0,\hat{\tau}_0)\|_{L^\infty_\xi})$, $U^\sigma=(u,\sigma)^\top$ and $U^\sigma_0=(u_0,\sigma_0)^\top$.
\end{lemma}
\begin{proof}
	First of all, by Plancherel's theorem, we have
	\begin{equation}\label{e4.41}
		\begin{split}
			&\|\nabla^k\left(\mathbb{G}_{u,\sigma}^1(t)\ast U_0^\sigma\right)\|_{L^2}
            \\ =&
            \left\|(|\xi|^k\mathcal{G}_3(\cdot,t)-\epsilon|\cdot|^{2+k}\mathcal{G}_1(\cdot,t))\hat{u}_0(\cdot) +  	\kappa|\cdot|^{1+k}\mathcal{G}_1(\cdot,t)\hat{\sigma}_0(\cdot)\right\|_{L^2}\\
			\geq & \left(\int_{|\xi|\leq R}|\xi|^{2k}\big|\mathcal{G}_3(\xi,t)\hat{u}_0(\xi) \big|^2 d\xi\right)^\frac12\\
            &-\left(\int_{|\xi|\leq R}\left(\epsilon|\xi|^{2+k}|\mathcal{G}_1(\xi,t)||\hat{u}_0(\xi)|+\kappa|\xi|^{1+k}|\mathcal{G}_1(\xi,t)||\hat{\sigma}_0(\xi)|\right)^2d\xi\right)^\frac12
            \\ =&
            I_1-I_2.
		\end{split}
	\end{equation}
Using the Lemma \ref{lemma_lower}, we have, for $t\geq t_1$,	
	\begin{equation}\label{I1}
	\begin{split}
		I_1 \geq &c_1\inf_{|\xi|\leq R}|\hat{u}_0(\xi)|\left(\int_{|\xi|\leq R}|\xi|^{2k}e^{-2\eta |\xi|^2t}d \xi\right)^\frac12\\
		 \geq& c_1c_0 \left(\int_{|\zeta|\leq \sqrt{t_1}R}t ^{-\frac32-k}e^{-2\eta |\zeta|^2}d \zeta  \right)^\frac12\gtrsim(1+t)^{-\frac{3}{4}-\frac{k}{2}}.
	\end{split}
	\end{equation}
On the other hand, applying Proposition \ref{theorem-Green} yields
	\begin{equation*}
    \begin{split}
		I_2 \leq &\|(\hat{u}_0,\hat{\sigma}_0)\|_{L^\infty_\xi}\left( \int_{|\xi|\leq R} \left( \epsilon |\xi|^{4+2k} +  \kappa^2|\xi|^{2+2k}\right)|\mathcal{G}_1(\xi,t)|^2 d \xi\right)^\frac12
        \\ \lesssim &
        \|(\hat{u}_0,\hat{\tau}_0)\|_{L^\infty_\xi} \left(\int_{|\xi|\leq R} |\xi|^{2+2k}e^{-2\theta|\xi|^2t}d\xi\right)^\frac12
        \\ \lesssim &
        (1+t)^{-\frac{5}{4}-\frac{k}{2}}.
	\end{split}
    \end{equation*}
Next, using Plancherel's theorem again, we are led to
	\begin{equation*}
	\begin{split}
	&\|\nabla^k\left(\mathbb{G}_{u,\sigma}^2(t)\ast U_0^\sigma\right)\|_{L^2}
\\=  & \left\|-\frac{\alpha}{2}|\cdot|^{1+k}\mathcal{G}_1(\cdot,t)\hat{u}_0(\cdot) +  [|\cdot|^k \mathcal{G}_2(\cdot,t)+\epsilon|\cdot|^{2+k}\mathcal{G}_1(\cdot,t)]\hat{\sigma}_0(\cdot)\right\|_{L^2}\\
\ge&\frac{\alpha}{2}\left(\int_{|\xi|\leq R}|\xi|^{2+2k}|\mathcal{G}_1(\xi,t)|^2|\hat{u}_0(\xi)|^2 d \xi\right)^\frac12\\
&  - \left(\int_{|\xi|\leq R}\left(|\xi|^k|\mathcal{G}_2(\xi,t)|+ \epsilon|\xi|^{2+k}|\mathcal{G}_1(\xi,t)|\right)^2|\hat{\sigma}_0(\xi)|^2 d \xi\right)^\frac12\\
	= & I_3 - I_4.
	\end{split}
	\end{equation*}
	Similiar to the analysis of $I_1$ and $I_2$, we have
	\begin{equation*}
	I_3 \gtrsim (1 + t)^{-\frac{5}{4}-\frac{k}{2}},
	\end{equation*}
	and
	\begin{equation}\label{I4}
	\begin{split}
	I_4 \lesssim& \|(\hat{u}_0, \hat{\sigma}_0)\|_{L^\infty_\xi} \left(\int_{|\xi|\leq R} \left(|\xi|^{4+2k}|\mathcal{G}_1(\xi,t)|^2+|\xi|^{2k}|\mathcal{G}_2(\xi,t)|^2\right) d \xi\right)^\frac12
	\\ \lesssim &
    \|(\hat{u}_0,\hat{\tau}_0)\|_{L^\infty_\xi}\left(\int_{|\xi|\leq R}  |\xi|^{4+2k}e^{-2\theta|\xi|^2 t}+e^{-\beta t} d \xi \right)^\frac12\lesssim (1 + t)^{-\frac{7}{4}-\frac{k}{2}}.
	\end{split}
	\end{equation}
It follows from \eqref{e4.41}-\eqref{I4} that \eqref{linear_lower_bound_1} and \eqref{linear_lower_bound_2} hold. The proof of Lemma \ref{linear_lower_bound} is complete.
\end{proof}
Based on the above analysis, we can get the lower decay bounds for the nonlinear system \eqref{system} in the following lemma.
\begin{lemma}\label{nonlinear_lower_bound}
	Under the assumptions of Part (ii) in Theorem \ref{theorem_decay}, there exist a positive time $t_0$ and a positive generic constant $c$, such that
	\begin{equation}\label{nonlinear_lower_bound_1}
\|\nabla^k u(t)\|_{L^2} \geq c(1 + t)^{-\frac{3}{4}-\frac{k}{2}}, \quad k=0,1,2,
	\end{equation}
	and
	\begin{equation}\label{nonlinear_lower_bound_2}
 \|\nabla^k \tau (t)\|_{L^2} \geq c(1 + t)^{-\frac{5}{4}-\frac{k}{2}}, \quad k=0, 1,
	\end{equation}
for all $t\geq t_0$ where $t_0=t_0(\alpha,\kappa,\beta,\|(\hat{u}_0,\hat{\tau}_0)\|_{L^\infty_\xi})$ and $c=c(c_2,C_2)$.
\end{lemma}
\begin{proof}
	First of all, by using the Duhamel's principle, we have
	\begin{equation*}
	\begin{split}
	u(t) = &\mathbb{G}^1_{u,\sigma}(t)\ast U_0^\sigma
	+
	\int_0^t
	\mathbb{G}^1_{u,\sigma}(t-s)\ast \left(\mathcal{F}_1,\mathcal{F}_2\right)^\top(s)
	{\rm d}s,\\
	\sigma(t)
	=  &
	\mathbb{G}^2_{u,\sigma}(t)\ast U_0^\sigma
	+
	\int_0^t
	\mathbb{G}^2_{u,\sigma}(t-s)\ast \left(\mathcal{F}_1,\mathcal{F}_2\right)^\top(s)
	{\rm d}s.
	\end{split}
	\end{equation*}
	It is easy to see that
	\begin{equation}\label{nonlinear_lower_1}
	\begin{split}
	\|\nabla^ku(t)\|_{L^2} =& \left\||\cdot|^k\hat{u}(t)\right\|_{L^2} \geq \left\||\cdot|^k\hat{u}(t)\right\|_{L^2_{|\xi|\leq R}}\\
	\geq&  \left\||\cdot|^k\hat{\mathbb{G}}^1_{u,\sigma}(t) \hat{U}_0^\sigma\right\|_{L^2_{|\xi|\leq R}} -
	\int_0^t \left\||\cdot|^k\hat{\mathbb{G}}^1_{u,\sigma}(t-s) \left(\hat{\mathcal{F}}_1,\hat{\mathcal{F}}_2\right)^\top(s)\right\|_{L^2_{|\xi|\leq R}}
	{\rm d}s,
	\end{split}
	\end{equation}
	and
	\begin{equation}\label{nonlinear_lower_2}
	\begin{split}
	\|\nabla^k\sigma(t)\|_{L^2} \geq&  \left\||\cdot|^k\hat{\mathbb{G}}^2_{u,\sigma}(t) \hat{U}_0^\sigma\right\|_{L^2_{|\xi|\leq R}} -
	\int_0^t \left\||\cdot|^k\hat{\mathbb{G}}^2_{u,\sigma}(t-s) \left(\hat{\mathcal{F}}_1,\hat{\mathcal{F}}_2\right)^\top(s)\right\|_{L^2_{|\xi|\leq R}}
	{\rm d}s.
	\end{split}
	\end{equation}
Similar to the proof of Lemma \ref{linear_lower_bound}, we have
	\begin{equation}\label{linear_lower_fb}
	\left\||\cdot|^k\hat{\mathbb{G}}^1_{u,\sigma}(t) \hat{U}_0^\sigma\right\|_{L^2_{|\xi|\leq R}} \gtrsim (1 + t)^{-\frac{3}{4}-\frac{k}{2}},~~   \left\||\cdot|^k\hat{\mathbb{G}}^2_{u,\sigma}(t) \hat{U}_0^\sigma\right\|_{L^2_{|\xi|\leq R}}\gtrsim (1 + t)^{-\frac{5}{4}-\frac{k}{2}}.
	\end{equation}
Now we  bound the nonlinear term in \eqref{nonlinear_lower_1} for $k=0, 1, 2$. In fact,	using Proposition \ref{theorem-Green}, Part (i) in Theorem \ref{theorem_decay} and Lemma \ref{nonlinear-estimates}, we obtain
	\begin{equation}\label{nonlinear_lower_3}
	\begin{split}
	&  \int_0^t \left\||\cdot|^k\hat{\mathbb{G}}^1_{u,\sigma}(t-s)\left(\hat{\mathcal{F}}_1,\hat{\mathcal{F}}_2\right)^\top(s) \right\|_{L^2_{|\xi|\leq R}}	{\rm d}s \\
	= &\int_0^t \|\left(|\cdot|^k\mathcal{G}_3(t-s)-\epsilon|\cdot|^{2+k}\mathcal{G}_1(t-s)\right)\hat{\mathcal{F}}_1 (s) + \kappa|\cdot|^{1+k}\mathcal{G}_1(t-s)\hat{\mathcal{F}}_2 (s)\|_{L^2_{|\xi|\leq R}}	{\rm d}s \\
	\lesssim &\int_0^{\frac{t}{2}} \left(\left\|\frac{1}{|\cdot|} \hat{\mathcal{F}}_1 (s)\right\|_{L^\infty_\xi}   + \|\hat{\mathcal{F}}_2 (s)\|_{L^\infty_\xi}\right)\left(\int_{|\xi|\leq R} |\xi|^{2(1+k)}e^{-2\theta|\xi|^2(t-s )} {\rm d} \xi\right)^{\frac{1}{2}} {\rm d}s \\
	&+ \int_{\frac{t}{2}}^t \left(\|(\hat{\mathcal{F}}_1 , \hat{\mathcal{F}}_2) (s)\|_{L^2_\xi}\right)(1+t-s)^{-\frac{k}{2}} {\rm d}s \\
	\lesssim &\int_0^{\frac{t}{2}}(1+s)^{-\frac{3}{2}}(1+t-s)^{-\frac{5}{4}-\frac{k}{2}}{\rm d}s + \int_0^{\frac{t}{2}}(1+s)^{-\frac{11}{4}}(1+t-s)^{-\frac{k}{2}}{\rm d}s\lesssim (1+t)^{-\frac{5}{4}-\frac{k}{2}}.
	\end{split}
	\end{equation}

Next, we turn to bound the nonlinear term in \eqref{nonlinear_lower_2} for $k=0, 1$.Similar to \eqref{nonlinear_lower_3}, using Proposition \ref{theorem-Green}, Part (i) in Theorem \ref{theorem_decay}, Lemma \ref{nonlinear-estimates} and \eqref{lemma_lower_eq2}, one deduces that
	\begin{equation}\label{nonlinear_lower_4}
\begin{split}
&  \int_0^t \left\||\cdot|^k\hat{\mathbb{G}}_{u,\sigma}^2(t-s) \left(\hat{\mathcal{F}}_1,\hat{\mathcal{F}}_2\right)^\top(s) \right\|_{L^2_{|\xi|\leq R}}	{\rm d}s \\
= &\int_0^t \left\| -\frac{\alpha}{2}|\cdot|^{k+1}\mathcal{G}_1(t-s)\hat{\mathcal{F}}_1 (s)+|\cdot|^k\left(\mathcal{G}_2(t-s) +\epsilon |\cdot|^2\mathcal{G}_1(t-s)\right) \hat{\mathcal{F}}_2 (s)\right\|_{L^2_{|\xi|\leq R}}	{\rm d}s\\
\lesssim &\int_0^{\frac{t}{2}} \left(\left\|\frac{1}{|\cdot|}\hat{\mathcal{F}}_1 (s)\right\|_{L^\infty_\xi} + \|\hat{\mathcal{F}}_2 (s)\|_{L^\infty_\xi}\right) \left(\int_{|\xi|\leq R} (|\xi|^{2(k+2)}e^{-2\theta|\xi|^2(t-s )} + e^{-\beta (t-s)}) {\rm d} \xi \right)^{\frac{1}{2}}{\rm d}s  \\
&+\int_{\frac{t}{2}}^t \left(\|( \hat{\mathcal{F}}_1 ,\hat{\mathcal{F}}_2 )(s,\cdot)\|_{L^2_\xi}\right)(1+t-s)^{-\frac{k}{2}} {\rm d}s\\
\lesssim&\int_0^{\frac{t}{2}}(1+s)^{-\frac{3}{2}}(1+t-s)^{-\frac{7}{4}-\frac{k}{2}}{\rm d}s+\int_{\frac{t}{2}}^t(1+s)^{-\frac{11}{4}}\left((1+t-s)^{-\frac{k}{2}}+e^{-\beta (t-s)}\right){\rm d}s\\
\lesssim& (1+t)^{-\frac{7}{4}-\frac{k}{2}}.
\end{split}
\end{equation}
Collecting \eqref{nonlinear_lower_1},\eqref{nonlinear_lower_2}, \eqref{linear_lower_fb}, \eqref{nonlinear_lower_3} and \eqref{nonlinear_lower_4}, and using the fact that $\|\nabla^k\sigma\|_{L^2}\lesssim \|\nabla^k\tau\|_{L^2},$ we conclude that \eqref{nonlinear_lower_bound_1} and \eqref{nonlinear_lower_bound_2} hold.

\end{proof}

With Lemma \ref{nonlinear_lower_bound}, the proof of Part (ii) in Theorem \ref{theorem_decay} is finished. Thus the proof of Theorem \ref{theorem_decay} is complete.

\section{Appendix}\label{sec5}
\setcounter{equation}{0}
\subsection{More time-decay estimates for vanishing center-of-mass diffusion}
\begin{theorem}\label{theorem appendix}
Under the conditions of Theorem \ref{theorem_decay}, for Case II (i.e., $\epsilon>0$, $\mu\ge0$), we get the optimal time-decay estimate for the third-order derivative of the solution, namely,
 \begin{equation}\label{app1}
\sum_{2\le k\le3}\|(\nabla^3u,\nabla^k\tau)(t)\|_{L^2}\le C(1+t)^{-\frac{9}{4}},
  \end{equation}
  and
\begin{equation}\label{eq_lower_case2}
\|\nabla^3 u(t)\|\geq c(1+t)^{-\frac94},\|\nabla^2 \tau(t)\|\geq c(1+t)^{-\frac94},
\end{equation}
for all $t\geq t_0,$ where $C$ and $c$ are positive constants independent of $t$.
\end{theorem}
\pf
To get (\ref{app1}), applying $\nabla^3$ to (\ref{system}), and then multiplying (\ref{system})$_1$ by $\alpha\nabla^3 u$ and (\ref{system})$_2$ by $\kappa\nabla^3 \tau$, we have from integration by parts and the cancellation relation that
\begin{equation} \label{3.2}
\begin{split}
     &\frac12\frac{d}{dt}\left(\alpha\|\nabla^3u\|_{L^2}^2+\kappa\|\nabla^3\tau\|_{L^2}^2\right)+\alpha\epsilon\|\nabla^4 u\|_{L^2}^2+\kappa\mu\|\nabla^4\tau\|_{L^2}^2+\kappa\beta\|\nabla^3\tau\|_{L^2}^2
     \\ =&
     -\alpha<\nabla^3(u\cdot\nabla u),\nabla^3 u>-\kappa<\nabla^3(u\cdot\nabla \tau),\nabla^3\tau>
     \\ &
     +\kappa<\nabla^3Q(\nabla u,\tau),\nabla^3\tau>=\sum_{p=1}^3\mathcal{K}_p,
\end{split}
\end{equation} where the incompressible condition, the H\"older's inequality and estimates on the commutator imply that,
\begin{equation*}
\begin{split}
    \mathcal{K}_1&=-\alpha\left(<\nabla^3(u\cdot\nabla u),\nabla^3 u>-<(u\cdot\nabla)\nabla^3 u,\nabla^3 u>\right)
    \\ & \lesssim
    \alpha\left\|\nabla^3(u\cdot\nabla u)-(u\cdot\nabla)\nabla^3 u\right\|_{L^2}\|\nabla^3u\|_{L^2}
    \lesssim
    \alpha\|\nabla u\|_{L^\infty}\|\nabla^3 u\|_{L^2}^2
    \\ & \lesssim
    \alpha\delta\|\nabla^3 u\|_{L^2}^2.
    \\
    \mathcal{K}_2&=-\kappa\left(<\nabla^3(u\cdot\nabla \tau),\nabla^3 \tau>-<(u\cdot\nabla)\nabla^3 \tau,\nabla^3 \tau>\right)
    \\ & \lesssim
    \left\|\nabla^3(u\cdot\nabla \tau)-(u\cdot\nabla)\nabla^3 \tau\right\|_{L^2}\|\nabla^3\tau\|_{L^2}
    \\ & \lesssim
    \kappa\left(\|\nabla u\|_{L^\infty}\|\nabla^3 \tau\|_{L^2}+\|\nabla\tau\|_{L^\infty}\|\nabla^3 u\|_{L^2}\right)\|\nabla^3\tau\|_{L^2}
    \\ & \lesssim
    \kappa\delta\left(\|\nabla^3 u\|_{L^2}^2+\|\nabla^3 \tau\|_{L^2}^2\right).
    \\
    \mathcal{K}_3&\lesssim \kappa\left(\|\nabla u\|_{L^\infty}\|\nabla^3\tau\|_{L^2}+\|\tau\|_{L^\infty}\|\nabla^4 u\|_{L^2}\right)\|\nabla^3\tau\|_{L^2}
    \\ & \lesssim
    \kappa\delta\|\nabla^3\tau\|_{L^2}^2+\kappa\delta\|\nabla^4 u\|_{L^2}\|\nabla^3\tau\|_{L^2}
    \\ & \leq
    \frac12\alpha\epsilon\|\nabla^4 u\|_{L^2}^2+C\left(\kappa\delta+\alpha^{-1}\epsilon^{-1}\kappa^2\delta^2\right)\|\nabla^3\tau\|_{L^2}^2.
\end{split}
\end{equation*}
In conclusion, by choosing $\delta=\delta(\kappa,\alpha,\beta,\epsilon)$ small enough, we have
\begin{equation*}
\begin{split}
     &\frac12\frac{d}{dt}\left(\alpha\|\nabla^3u\|_{L^2}^2+\kappa\|\nabla^3\tau\|_{L^2}^2\right)+\frac12\alpha\epsilon\|\nabla^4 u\|_{L^2}^2+\kappa\mu\|\nabla^4\tau\|_{L^2}^2+\frac12\kappa\beta\|\nabla^3\tau\|_{L^2}^2
     \\ \leq &
    C\delta(\alpha+\kappa)\|\nabla^3 u\|_{L^2}^2.
\end{split}
\end{equation*} Then we have from a similar discussion as (\ref{4.34}) that
\begin{equation*}
\begin{split}
\frac{\rm d}{{\rm d}t}\|\nabla^3(u,\tau)\|_{L^2}^2+D_2\|\nabla^3(u,\tau)\|_{L^2}^2
\lesssim
\int_{|\xi|\leq R}
|\xi|^6|\hat u(t)|^2
{\rm d}\xi,
\end{split}
\end{equation*} for some positive constant $D_2=D_2(\alpha,\kappa,\beta,\epsilon)$. Then, it follows from Lemma \ref{nonlinear-estimates} that
\begin{equation}\label{4.363}
\begin{split}
\|\nabla^3(u,\tau)(t)\|_{L^2}^2
\lesssim
&
{\rm e}^{-D_2 t}
\|\nabla^3U_0\|_{L^2}^2
+
\int_0^t
{\rm e}^{-D_2(t-s)}
\left(
\int_{|\xi|\leq R}
|\xi|^6|\hat{u}(s)|^2
{\rm d}\xi
\right)
{\rm d}s\\
\lesssim
&
{\rm e}^{-D_2t}
\|\nabla^3U_0\|_{L^2}^2
+
\int_0^t
{\rm e}^{-D_2(t-s)}
(1+s)^{-\frac92}
\|\hat{U}_0\|_{L^\infty_\xi}^2
{\rm d}s
\\ &
+\int_0^t {\rm e}^{-D_2(t-s)}\Bigg[ \int_0^{\frac{s}{2}}(1+s-s^\prime)^{-\frac92}\left\|\left(\mathcal{M}_1,\mathcal{M}_2\right)^\top(s)\right\|_{L^1}{\rm d}s^\prime \\
&+ \int_{\frac{s}{2}}^s(1+s-s^\prime)^{-3}\left\|\left(\mathcal{M}_1,\mathcal{M}_2\right)^\top(s)\right\|_{L^2}{\rm d}s^\prime\Bigg]{\rm d}s
\\
\leq
&
C_2(1+t)^{-\frac{9}{2}},
\end{split}
\end{equation}
which implies that
\begin{equation}\label{3rd}
\|\nabla^3(u,\tau)(t)\|_{L^2}\leq C_2(1+t)^{-\frac94}.
\end{equation}

Finally, \eqref{3rd} enables us to improve the decay rate of $\|\nabla^2\tau\|_{L^2}$ in Case II. In fact, similar to \eqref{eq_tau}, and using \eqref{3rd}, we have
	 \begin{equation*}
	 	\begin{split}
	 	&\frac{1}{2}\frac{\mathrm{d}}{\mathrm{d}t} \|\nabla^2 \tau\|_{L^2}^2 + \mu \|\nabla^{3} \tau \|_{L^2}^2 + \frac{\beta}{2}\|\nabla^2 \tau\|_{L^2}^2\\
	 	\lesssim & \alpha \|\nabla^{3}u\|_{L^2}^2 + \|\nabla^2{Q}(\nabla u,\tau)\|_{L^2}^2 + \|\nabla^{2}(u\cdot \nabla\tau)\|_{L^2}^2\\
	 	\lesssim & \alpha \|\nabla^{3}u\|_{L^2}^2 + \|\nabla^{3}u\|_{L^2}^2\|\tau\|_{L^\infty}^2 + \|\nabla  u\|_{L^\infty}^2\|\nabla^2\tau\|_{L^2}^2\\&+ \|\nabla^{3}\tau\|_{L^2}^2\|u\|_{L^\infty}^2 + \|\nabla \tau\|_{L^\infty}^2\|\nabla^{2}u\|_{L^2}^2\\
	 	\lesssim & (1+ t)^{-\frac92}.
	 	\end{split}
	\end{equation*}
	Consequently,
	\begin{equation}\label{2nd-tau}
		\|\nabla^2 \tau\|_{L^2} \leq C_2 (1+t)^{-\frac{9}{4}}.
	\end{equation}

Next, we will justify (\ref{eq_lower_case2}). Taking $k =3$ in the nonlinear term of \eqref{nonlinear_lower_1}, we have
\begin{equation}\label{nonlinear_lower_u3}
\begin{split}
&  \int_0^t \||\xi|^3\hat{\mathbb{G}}^1_{u,\sigma}(t-s)\left(\hat{\mathcal{F}}_1,\hat{\mathcal{F}}_2\right)^\top(s) \|_{L^2(|\xi|\leq R)}	{\rm d}s \\
= &\int_0^t \|\left(|\cdot|^3\mathcal{G}_3(t-s)-\epsilon|\cdot|^{5}\mathcal{G}_1(t-s)\right)\hat{\mathcal{F}}_1 (s) + \kappa|\cdot|^{4}\mathcal{G}_1(t-s)\hat{\mathcal{F}}_2 (s)\|_{L^2(|\xi|\leq R)}	{\rm d}s \\
\lesssim &\int_0^{\frac{t}{2}} \left(\left\|\frac{1}{|\cdot|} \hat{\mathcal{F}}_1 (s)\right\|_{L^\infty_\xi}   + \|\hat{\mathcal{F}}_2 (\cdot,s)\|_{L^\infty_\xi}\right)\left(\int_{|\xi|\leq R} |\xi|^{8}e^{-2\theta|\xi|^2(t-s )} {\rm d} \xi\right)^{\frac{1}{2}} {\rm d}s \\
 &+\int_{\frac{t}{2}}^t \left(\|( \hat{\mathcal{F}}_1 ,\hat{\mathcal{F}}_2 )(\cdot,s)\|_{L^2_\xi}\right)(1+t-s)^{-\frac{3}{2}} {\rm d}s \\
\lesssim &\int_0^{\frac{t}{2}}(1+s)^{-\frac{3}{2}}(1+t-s)^{-\frac{11}{4}}{\rm d}s + \int_{\frac{t}{2}}^t(1+s)^{-\frac{5}{2}}(1+t-s)^{-\frac{3}{2}}{\rm d}s \lesssim (1+t)^{-\frac{5}{2}}.
\end{split}
\end{equation}
Then taking $k = 2$ in the nonlinear term of \eqref{nonlinear_lower_2}, we have
\begin{equation}\label{nonlinear_lower_4tau}
\begin{split}
&  \int_0^t \left\||\cdot|^2\hat{\mathbb{G}}_{u,\sigma}^2(t-s) \left(\hat{\mathcal{F}}_1,\hat{\mathcal{F}}_2\right)^\top(s) \right\|_{L^2_{|\xi|\leq R}}	{\rm d}s \\
= &\int_0^t \left\| -\frac{\alpha}{2}|\cdot|^3\mathcal{G}_1(t-s)\hat{\mathcal{F}}_1 (s)+\left(|\cdot|^2\mathcal{G}_2(t-s) +\epsilon |\cdot|^4\mathcal{G}_1(t-s)\right) \hat{\mathcal{F}}_2 (s)\right\|_{L^2_{|\xi|\leq R}}	{\rm d}s\\
\lesssim &\int_0^{\frac{t}{2}} \left(\left\|\frac{1}{|\cdot|}\hat{\mathcal{F}}_1 (s)\right\|_{L^\infty_\xi} + \|\hat{\mathcal{F}}_2 (s)\|_{L^\infty_\xi}\right) \left(\int_{|\xi|\leq R} (|\xi|^{6}e^{-2\theta|\xi|^2(t-s )} + e^{-\beta (t-s)}) {\rm d} \xi \right)^{\frac{1}{2}}{\rm d}s  \\
&+\int_{\frac{t}{2}}^t \left(\|( \hat{\mathcal{F}}_1 ,\hat{\mathcal{F}}_2 )(\cdot,s)\|_{L^2_\xi}\right)(1+t-s)^{-1} {\rm d}s\\
\lesssim&\int_0^{\frac{t}{2}}(1+s)^{-\frac{3}{2}}(1+t-s)^{-\frac{11}{4}}{\rm d}s+\int_{\frac{t}{2}}^t(1+s)^{-\frac{5}{2}}\left((1+t-s)^{-1}+e^{-\beta (t-s)}\right){\rm d}s\\
\lesssim& (1+t)^{-\frac{5}{2}}.
\end{split}
\end{equation}
Combining \eqref{nonlinear_lower_1},\eqref{nonlinear_lower_2},\eqref{linear_lower_fb},\eqref{nonlinear_lower_u3} and \eqref{nonlinear_lower_4tau}, using the fact that $\|\nabla^2\sigma\|_{L^2}\lesssim \|\nabla^2\tau\|_{L^2},$ we conclude that \eqref{eq_lower_case2} holds.

\section*{Acknowledgements}
J. Huang is partially supported by the National Natural Science Foundation of China (Grant No. 1197135\\7,  11771155 and 11571117), the Natural Science Foundation of Guangdong
Province (Grant No. 2019A1\\515011491), the Innovation Project of Department of Education of Guangdong Province of China (Grant No. 2019KTSCX183).
H. Wen is supported by the
National Natural Science Foundation of China (Grant No. 11722104, 11671150). R. Zi is partially  supported by the National Natural Science Foundation of China (Grant No. 11871236 and 11971193), the Natural Science Foundation of Hubei Province (Grant No. 2018CFB665), and the Fundamental Research Funds for the Central Universities (Grant No. CCNU19QN084).


\begin{thebibliography}{99}




\bibitem{Barrett-Boyava-M3-11}
J.W. Barrett, S. Boyaval, Existence and approximation of a (regularized) Oldroyd-B model. {\it Math. Models Methods Appl. Sci.}, 21(2011), 1783--1837.

\bibitem{Barrett-Lu-Suli}
J.W. Barrett, Y. Lu and E. S${\rm \ddot{u}}$li, Existence of large-data finite-energy global weak solutions to a compressible Oldroyd-B model. {\it Commun. Math. Sci.}, 15(2017), 1265--1323.

\bibitem{Barrett-Suli-M3-11}
J.W. Barrett, E. S${\rm \ddot{u}}$li, Existence and equilibration of global weak solutions to kinetic models for dilute polymers I: Finitely extensible nonlinear bead-spring chains. {\it Math. Models Methods Appl. Sci.}, 21(2011), 1211--1289.

\bibitem{Barrett_Suli_2018}J.W. Barrett, E. S${\rm \ddot{u}}$li, Existence of global weak solutions to the kinetic Hookean dumbbell model for incompressible dilute polymeric fluids. {\it Nonlinear Anal. Real World Appl.}, 39 (2018), 362--395.


\bibitem{Bathory_etal} M. Bathory, M. Bul\'{i}\v{c}ek, J. M\'{a}lek, Large data existence theory for three-dimensional unsteady flows of rate-type viscoelastic fluids with stress diffusion. {\it Adv. Nonlinear Anal.},  10 (2021), no. 1, 501--521.

\bibitem{Bhave1}A.V. Bhave, R.C. Armstrong, R.A. Brown, Kinetic theory and rheology of dilute,
nonhomogeneous polymer solutions. {\it J. Chem. Phys.}, 95(1991), 2988--3000.

\bibitem{Bhave2}A.V. Bhave, R.K. Menon, R.C. Armstrong, R.A. Brown, A constitutive equation for liquid-crystalline polymer solutions. {\it J. Rheol.}, 37(3)(1993), 413--441.

\bibitem{Boyaval_etal_2009} S. Boyaval, T. Leli\`{e}vre, C.  Mangoubi, Free-energy-dissipative schemes for the Oldroyd-B model. {\it M2AN Math. Model. Numer. Anal.} 43 (2009), 523--561.

\bibitem{Cai}Y. Cai, Z. Lei, F.H. Lin, N. Masmoudi, Vanishing viscosity limit for incompressible
viscoelasticity in two dimensions. {\it Comm. Pure Appl. Math.}, 72(2019), 2063--2120.

\bibitem{Cates_2006} M. E. Cates, S. M. Fielding, Rheology of giant micelles. {\it  Adv. Phys.}, 55(2006), 799--879.

\bibitem{Chen}Q.L. Chen, C.X. Miao, Global well-posedness of viscoelastic fluids of Oldroyd type in
Besov spaces. {\it Nonlinear Anal.}, 68(2008), 1928--1939.

\bibitem{Chupin}L. Chupin, S. Martin, Viscoelastic flows in a rough channel: Amultiscale analysis. {\it Ann. I.H.Poincar\'e-AN}, 34(2017), 483--508.

\bibitem{Constantin-Kliegl} P. Constantin, M. Kliegl, Note on global regularity for two dimensional Oldroyd-B fluids stress. {\it Arch.
Ration. Mech. Anal.}, 206 (2012), 725--740.


\bibitem{Constantin-3} P. Constantin, J.H. Wu, J.F. Zhao, Y. Zhu, High Reynolds number and high Weissenberg number Oldroyd-B model with dissipation, arXiv: 2001.03703 [math.AP], 2020.

\bibitem{Dhont_2008}J. K. G. Dhont,     W. J. Briels,  Gradient and vorticity banding,  {\it Rheol. Acta} 47(3)(2008), 257--281.
Acta 47(3), 257–281 (2008).

\bibitem{Dostalik_etal_2020} M. Dostal\'{i}k, V. Pr\r{u}\v{s}a, J. Stein, Unconditional finite amplitude stability of a viscoelastic fluid in a mechanically isolated vessel with spatially non-uniform wall temperature. {\it Math. Comput. Simulation}, 2020.

\bibitem{E_Li_Zhang_2004} W.N. E, T.J. Li, P.W. Zhang, Well-posedness for the dumbbell model of polymeric
fluids. {\it Comm. Math. Phys.} 248 (2004), 409--427.

\bibitem{Elgindi-Liu} T.M. Elgindi, J.L. Liu, Global wellposedness to the generalized Oldroyd type models in $\mathbb{R}^3$. {\it J. Differential Equations}, 259(5) (2015), 1958--1966.

\bibitem{Elgindi-Rousset} T.M. Elgindi, F. Rousset, Global regularity for some Oldroyd-B type models. {\it Comm. Pure Appl. Math.}, 68(11)(2015), 2005--2021.

\bibitem{El-Kareh} A.W. El-Kareh, L.G. Leal, Existence of solutions for all Deborah numbers for a non-Newtonian model modified to include diffusion. {\it J. Non-Newton. Fluid Mech.}, 33(1989), 257--287.

\bibitem{Fang-Hieber-Zi} D.Y. Fang, M. Hieber, R.Z. Zi, Global existence results for Oldroyd-B Fluids in
exterior domains: The case of non-small coupling parameters. {\it Math. Ann.}, 357(2013), 687--709.

\bibitem{Fang-Zi}D.Y. Fang, R.Z. Zi, Global solutions to the Oldroyd-B model with a class of large initial
data. {\it SIAM J. Math. Anal.}, 48(2016), 1054--1084.

\bibitem{Fernandez-Guillpen-Ortega}
E. Fern${\rm \acute{a}}$ndez-Cara, F. Guill${\rm \acute{p}}$en and R. Ortega, Some theoretical results concerning non-Newtonian fluids of the Oldroyd kind. {\it Ann. Scuola Norm. Sup. Pisa}, 26 (1998), 1--29.


\bibitem{Guillop-1} C. Guillop\'e, J.C. Saut. Existence results for the flow of viscoelastic fluids with a differential
constitutive law. {\it Nonlinear Anal., Theory, Methods Appl.}, 15(1990), 849--869.

\bibitem{Hall_2015} B. Hall, Lie Groups, Lie Algebras, and Representations, An Elementary Introduction. Second edition. {\it Graduate Texts in Mathematics,}  Springer, Cham, 2015.

\bibitem{Hieber-Naito-Shibata}
M. Hieber, Y. Naito, Y. Shibata, Global existence results for Oldroyd-B fluids in exterior domains. {\it J. Differential Equations}, 252(2012), 2617--2629.

\bibitem{Hieber-Wen-Zi} M. Hieber, H.Y. Wen, R.Z. Zi, Optimal decay rates for solutions to the incompressible Oldryod-B model in $\mathbb{R}^3$. {\it Nonlinearity}, 32(3)(2019), 833--852.

\bibitem{Hu_Lelievre}
D. Hu, T. Leli\`{e}vre, New entropy estimates for Oldroyd-B and related models. {\it Commun. Math. Sci.}, 5 (2007), no. 4, 909--916.


\bibitem{Hu-Lin}
X.P. Hu, F.H. Lin, Global solutions of two-dimensional incompressible viscoelastic flows with discontinuous initial data. {\it Comm. Pure Appl. Math.}, LXIX (2016) 0372--0404.

\bibitem{Hu-Wu}X.P. Hu, H. Wu, Long-time behavior and weak-strong uniqueness for incompressible
viscoelastic flows. {\it Discrete Continuous Dyn. Syst.}, 35(2015), 3437--3461.

\bibitem{Hu-Wu_2013} X.P. Hu, G.C. Wu. Global existence and optimal decay rates for three-dimensional compressible viscoelastic flows. {\it SIAM J. Math. Anal.}, 45(2013), 2815--2833.

\bibitem{Kawashima1983}
S.~Kawashima,
\newblock {Systems of a hyperbolic-parabolic composite type, with applications to the equations of magnetohydrodynamics.}
\newblock {\it Doctoral Thesis Kyoto Univ.}, 1983.

\bibitem{La} J. La, On diffusive 2D Fokker-Planck-Navier-Stokes systems. {\it Arch. Rational Mech. Anal.}, 235(2020), 1531--1588.

\bibitem{Lai} B.S. Lai, J.Y. Lin, C.Y. Wang, Forward self-similar solutions to the viscoelastic Navier-Stokes equation with damping. {\it SIAM J. Math. Anal.}, 49(2017), no. 1, 501--529.


\bibitem{Lei1}Z. Lei, On 2D viscoelasticity with small strain. {\it Arch. Ration. Mech. Anal.}, 198(2010), 13--37.

\bibitem{Lei2} Z. Lei, C. Liu, Y. Zhou, Global solutions for incompressible viscoelastic fluids. {\it Arch. Ration.
Mech. Anal.}, 188(2008), 371--398.


\bibitem{Lin} F.H. Lin, Some analytical issues for elastic complex fluids. {\it Comm. Pure Appl. Math.}, 65(2012), 893--919.

\bibitem{Lin-Liu-Zhang}
F.H. Lin, C. Liu and P. Zhang, On hydrohynamics of viscoelastic fluids. {\it Comm. Pure Appl. Math.}, 58(2005) 1437--1471.

\bibitem{Lin-Zhang}F.H. Lin, P. Zhang, On the initial-boundary value problem of the incompressible viscoelastic
fluid system. {\it Commun. Pure Appl. Math.}, 61(2008), 539--558.

\bibitem{Lions-Masmoudi}
P.L. Lions, N. Masmoudi, Global solutions for some Oldroyd models of non-Newtonian flows. {\it Chinese Ann. Math. Ser. B}, 21(2000), 131--146.

\bibitem{Liu} A.J. Liu, G.H. Fredrickson, Free energy functionals for semiflexible polymer solutions and blends. {\it Macromolecules}, 26(11)(1993), 2817--2824.

\bibitem{Lu-Zhang}
Y. Lu, Z.F. Zhang, Relative entropy, weak-strong uniqueness and conditional regularity
for a compressible Oldroyd-B model. {\it SIAM J. Math. Anal.}, 50(1)(2018), 557--590.



\bibitem{Malek_etal_2018}
J. M\'{a}lek, V. Prů\v{s}a, T. Sk\v{r}ivan, E. S\"{u}li, Thermodynamics of viscoelastic rate-type fluids with stress diffusion. {\it Physics of Fluids}, 30(2018), 023101.

\bibitem{Molinet-Talhouk}
L. Molinet, R. Talhouk, On the global and periodic regular flows of viscoelastic fluids with a differential constitutive law, {\it Nonlinear Diff. Equations Appl.}, 11(2004), 349--359.

\bibitem{Oldroyd} J. Oldroyd, Non-Newtonian effects in steady motion of some idealized elasticoviscous liquids. {\it Proc. Roy. Soc. Edinburgh Sect. A}, 245(1958), 278--297.

\bibitem{Rajagopal_2000}K. R. Rajagopal, A. R. Srinivasa, A thermodynamic frame work for
rate type fluid models, {\it J. Non-Newtonian Fluid Mech.} 88(3)(2000), 207–-227.

\bibitem{Schonbek1} M.E. Schonbek, $L^2$ decay for weak solutions of the Navier-Stokes equations. {\it Arch. Ration. Mech. Anal.}, 88(3)(1985), 209--222.

\bibitem{Schonbek2} M.E. Schonbek, Large time behavior of solutions to the Navier-Stokes equations. {\it Comm. Partial Differential
Equations}, 11(1986), 753--763.

\bibitem{Schonbek_Suli_1996}
M.E. Schonbek, T.P. Schonbek,  E. S\"uli, Large-time behaviour of solutions to the magnetohydrodynamics equations. {\it Math. Ann.}, 304(1996), no.4, 717--756.

\bibitem{Wang-Wen} W.J. Wang, H.Y. Wen, The Cauchy problem for an Oldroyd-B model in three dimensions. {\it Math. Models Methods Appl. Sci.}, 30(2020), 139--179.

\bibitem{Zhang-Fang} T. Zhang, D.Y. Fang, Global existence of strong solution for equations related to the incompressible viscoelastic fluids in the critical $L^p$ framework. {\it SIAM J. Math. Anal.}, 44(2012), 2266--2288.


\bibitem{Zhu} Y. Zhu, Global small solutions of 3D incompressible Oldroyd-B model without damping mechanism. {\it J. Functional Analysis}, 274(7)(2018), 2039--2060.


\bibitem{Fang-Zhang-Zi} R.Z. Zi, D.Y. Fang, T. Zhang, Global solution to the incompressible Oldroyd-B model in the critical $L^p$ framework: the case of the non-small coupling parameter. {\it Arch. Ration. Mech. Anal.}, 213(2)(2014), 651--687.

\bibitem{Ziegler_1987}H. Ziegler, C. Wehrli, The derivation of constitutive relations from the
free energy and the dissipation function,  {\it Adv. Appl. Mech.} 25(1987), 183--238.

\end{thebibliography}
\end{document}